\documentclass[12pt,reqno]{amsart}
\usepackage{amssymb}
\usepackage[curve,matrix,arrow]{xy}
\newtheorem{thm}{Theorem}[section]
\newtheorem{lemma}[thm]{Lemma}
\newtheorem{cor}[thm]{Corollary}
\newtheorem{prop}[thm]{Proposition}

\newtheorem{notation}[thm]{Notation}

\theoremstyle{definition}
\newtheorem{rem}[thm]{Remark}

\newcommand{\qbox}[1]{\quad\hbox{#1}\quad}

\def\fS{{\mathfrak{S}}}
\def\bZ{{\mathbb Z}}
\def\bF{{\mathbb F}}

\def\Hom{\operatorname{Hom}\nolimits}
\def\End{\operatorname{End}\nolimits}

\def\Inf{\operatorname{Inf}\nolimits}

\def\proj{\operatorname{(proj)}\nolimits}

\def\Det{\operatorname{Det}\nolimits}

\def\SL{\operatorname{SL}\nolimits}
\def\GL{\operatorname{GL}\nolimits}
\def\PSL{\operatorname{PSL}\nolimits}
\def\PGL{\operatorname{PGL}\nolimits}
\def\SP#1#2{\operatorname{Sp}\nolimits}

\def\bZ{{\mathbb Z}}
\def\bfq{{\mathbb F}_q}
\def\whg{{\widehat G}}
\def\whn{{\widehat N}}
\def\whl{{\widehat L}}
\def\whh{{\widehat H}}
\def\whs{{\widehat S}}

\def\whj{{\widehat J}}
\def\whz{{\widehat Z}}

\def\gcd{\operatorname{gcd}}

\newcommand{\ls}[2]{{}^{{#1}}\!{{#2}}}

\title[Endotrivial modules for finite groups of Lie type $A$]
{Endotrivial modules for finite groups of Lie type $A$ in 
nondefining characteristic}

\author{\sc Jon F. Carlson}
\address
{Department of Mathematics\\ University of Georgia \\
Athens\\ GA~30602, USA}
\thanks{Research of the first author was supported in part by NSF
grant DMS-1001102}
\email{jfc@math.uga.edu}
\author{\sc Nadia Mazza}
\address
{Department of Mathematics and Statsitics\\ University of
  Lancaster\\ Lancaster \\LA1 4YF, UK} 
\email{n.mazza@lancaster.ac.uk}
\author{\sc Daniel K. Nakano}
\address
{Department of Mathematics\\ University of Georgia \\
Athens\\ GA~30602, USA}
\thanks{Research of the third author was supported in part by NSF
grant DMS-1402271}
\email{nakano@math.uga.edu}
\date{\today}

\begin{document}

\begin{abstract}
Let $G$ be a finite group such that 
$\SL(n,q)\subseteq G \subseteq \GL(n,q)$ and
$Z$ be a central subgroup of $G$. In this 
paper we determine the  
group $T(G/Z)$ consisting of the equivalence classes 
of endotrivial $k(G/Z)$-modules 
where $k$ is an algebraically closed field of characteristic 
$p$ such that $p$ does not divide $q$. 
The results in this paper complete the classification of 
endotrivial modules for all finite groups of Lie Type $A$, initiated
earlier by the authors. 
\end{abstract}
\maketitle

\section{Introduction}\label{S:intro} 
Let $G$ be a finite group and $k$ be a
field of characteristic $p> 0$.
The group of endotrivial $kG$-modules was first
introduced for $p$-groups by Dade \cite{D1,D2} nearly
forty years ago. He showed that the endotrivial
modules for a Sylow $p$-subgroup $S$ of $G$
are the building blocks of the endo-permutation $kS$-modules which are
the sources of the irreducible $kG$-modules when the group $G$ is
$p$-nilpotent. 
For any finite group $G$,
tensoring with an endotrival $kG$-module induces a
self-equivalence on the stable category of $kG$-modules
modulo projectives. Thus the group of endotrivial modules
is an important part of the Picard group of self-equivalences
of the stable category, namely, the self-equivalences
of Morita type. In addition, the endotrivial modules
are the modules whose deformation rings are universal
and not just versal (see \cite{BlCh}).

The endotrivial modules for an abelian  $p$-group were
classified by Dade and a complete classification of 
endotrivial modules
over any $p$-group was completed several
years later by the first author and
Th\'evenaz \cite{CT2, CT3, Ca}  building on the work of Alperin
\cite{Al} and others. Since then there has been an effort
to compute the group $T(G)$ of endotrivial modules for almost simple and
quasi-simple groups $G$. The proofs of \cite{CMT} suggest that this
might be an important step in the computation of $T(G)$ for an arbitrary
finite group $G$. The group $T(G)$ has been
determined for finite groups of
Lie type in the defining characteristic in \cite{CMN1},
and for symmetric and alternating groups in
\cite{CMN2, CHM}. Other results can
be found in \cite{CMT,LMS,NR,Rob}.

Every one of the papers in this project has produced 
important advances for computing and determining
endotrivial modules. This paper continues that development,
presenting a significant improvement of a method introduced
in \cite{CMN3}. 
The method was inspired by the development
by Balmer \cite{Bal} of ``weak $H$-homomorphisms" which
describe the kernel of the restriction map $T(G) \to T(H)$
when $H$ is a subgroup of $G$ that contains a Sylow
$p$-subgroup of $G$.
The new technique with some variations
allows the computations of the group of endotrivial modules for all
finite groups of Lie type $A$ in nondefining characteristic.
In all but a few examples of small Lie rank and small 
characteristic we show that the torsion part of $T(G)$
is generated by the isomorphism classes of
one-dimensional modules.  There are a couple of instances
in this paper when it is necessary to call upon a somewhat
more sophisticated variation of the method developed in
\cite{CT4}.

Hence, the goal of this paper is to describe $T(G)$ for all
finite groups of Lie type $A$ in nondefining characteristic, completing
the work started in \cite{CMN3}. The general result is the following. 
The structure of the groups of endotrivial modules for cases not 
covered by this theorem are treated in later sections. 

\begin{thm} \label{thm:noncyclic1a} 
Let $k$ be an algebraically closed field of 
prime characteristic $p$ and
$q$ a prime power with $p$ not dividing $q$, 
and let $e$ be the least positive
integer such that $p$ divides $q^e-1$.
Let $G$ be a finite group of order divisible by $p$ such that 
$\SL(n,q)\subseteq G \subseteq \GL(n,q)$, and let $Z$ be a central
subgroup of $G$.
Assume that the following
conditions hold. 
\begin{itemize} 
\item[(a)]  In all cases, $n \geq 2e$.
\item[(b)]  If $e=1$, $n=2$ and $p\geq 3$, then $Z$ does not contain
a Sylow $p$-subgroup of $Z(G)$.
\item[(c)]  If $e = 1$ and $n=p=3$, then $Z$ does not 
contain a Sylow $3$-subgroup of $Z(G)$
(which happens if and only if and only if $3$ divides $\vert G/(Z\cdot
\SL(3,q)) \vert$).   
\item[(d)]  If $p = 2$, then $n > 3$. 
\end{itemize}
Then 
\[
T(G/Z) \cong \bZ \oplus X(G/Z),
\] 
where $X(G/Z)$ is the group under tensor product of 
$k(G/Z)$-modules of dimension one and the torsion free part 
of $T(G)$ is 
generated by the class of $\Omega(k)$.  
\end{thm}

Theorem~\ref{thm:noncyclic1a} is established in Sections \ref{sec:SLncase}
and \ref{sec:thm11}. This follows some preliminaries on endotrivial
modules in Section \ref{sec:prelim-endo}, a description of the main
method that we use in most proofs in Section \ref{sec:maintec}, and
preliminaries on groups of Lie type A in Section \ref{sec:group-structure}.
The proof of the main theorem stated above 
is accomplished in two major steps. In Section 
\ref{sec:SLncase}, we treat the case that $G = \SL(n,q)$ and 
$Z = \{1\}$. In Section \ref{sec:thm11}, the 
result is extended to any $G$ and $Z$ subject to the assumptions of
Theorem~\ref{thm:noncyclic1a}. 

Sections~\ref{sec:gl2} through ~\ref{sec:appendix} deal with the 
cases that are excluded by the hypotheses of
Theorem~\ref{thm:noncyclic1a}. 
In the nontrivial cases excluded by condition (a), namely, 
$e \leq n<2e$, the Sylow $p$-subgroup of $G$ is cyclic and the structure of
$T(G)$ was provided in \cite[Theorem 1.2]{CMN3}. For the sake
completeness, the theorem is stated in an appendix
(cf. Theorem~\ref{thm:cyclic}).  The one additional case 
in which the Sylow $p$-subgroup of $G/Z$ is cyclic is the case 
excluded by hypothesis (b) of Theorem \ref{thm:noncyclic1a}. 
This case is dealt with in Theorem~\ref{thm:sl2case}. 

The case excluded from Theorem~\ref{thm:noncyclic1a} by
condition~(c), is treated in Section \ref{sec:char3}. 
In Sections~\ref{sec:sl32} and \ref{sec:sl22}
we compute of the groups of endotrivial
modules when $p=2$ and $n = 2$ or $3$, excluded from
Theorem~\ref{thm:noncyclic1a} by condition~(d).  These sections 
use results of \cite{CT4}, that show the existence of exotic
endotrivial modules in a few cases. By ``exotic" we mean trivial source
modules that do not have dimension one.  

Our results for the nondefining characteristic, taken together with the results in \cite{CMN1} 
(for the defining characteristic), provide a complete description of the group of endotrivial modules for finite groups of 
Lie type $A$ over algebraically closed fields of arbitrary characteristic.

%%%%%%%%%%%%%%%%%%%%%  section 2 %%%%%%%%%%%%%%%%%%%%%%%%%%%%

\section{Endotrivial Modules} \label{sec:prelim-endo}

Throughout the paper, let $k$ be an algebraically 
closed field of prime characteristic~$p$ and 
$G$ be a finite group with $p$ 
dividing the order of $G$. 
All $kG$-modules in this paper are assumed to be 
finitely generated.
For $kG$-modules $M$ and $N$, let $M^*=\Hom_k(M,k)$ 
denote the $k$-dual of $M$ and write 
$M\otimes N=M\otimes_kN$. The modules $M^*$ and
$M\otimes N$ become $kG$-modules under the 
the usual Hopf algebra structure on $kG$.

A $kG$-module $M$ is {\em endotrivial} provided its endomorphism algebra 
$\End_k(M)$ splits as the direct sum of $k$ and a projective
$kG$-module. That is, since $\Hom_k(M, N) \cong M^*\otimes N$, as
$kG$-modules, $M$ is endotrivial if and only if  
\[
\End_k(M)  \cong  M^*\otimes M \cong k \ \oplus \ P 
\]
for some projective $kG$-module $P$.

Any endotrivial $kG$-module $M$ has a unique 
indecomposable nonprojective endotrivial 
direct summand $M_0$ (\cite{CMN1}). 
This allows us to define an equivalence relation 
on the class of endotrivial $kG$-modules;
namely, two endotrivial $kG$-modules are equivalent if they have
isomorphic indecomposable nonprojective summands. That is, two endotrivial $kG$-modules are 
equivalent if they are isomorphic in the stable category.
The set of equivalence classes of endotrivial $kG$-modules is an abelian
group with the operation induced by the tensor of product over $k$,
$$
[M]+[N]=[M\otimes N].
$$  
The identity element of $T(G)$ is $[k]$, and the inverse of $[M]$ is
$[M^* ]$. The group $T(G)$ is called the {\em group of endotrivial
$kG$-modules}. 

It is well-known that the group of endotrivial modules  is 
a finitely generated abelian group.   Therefore, 
\[
T(G) \cong TF(G) \ \oplus \ TT(G)
\] 
where $TT(G)$ is the torsion subgroup of $T(G)$ 
and $TF(G)$ is a torsion
free complement. The rank 
of $TF(G)$  depends only on the $p$-local structure of $T(G)$,
as described in the next theorem. 
Recall that the $p$-rank of a group is the maximum of the ranks
of elementary abelian $p$-subgroups of $G$, and a maximal elementary
abelian $p$-subgroup is an elementary abelian $p$-subgroup which is not
properly contained in any other elementary abelian $p$-subgroup. 
Let $n_G$ be the number of conjugacy classes of
maximal elementary abelian $p$-subgroups of $G$ of order $p^2$.

\begin{thm}\label{thm:rank} \cite[Theorem 3.1]{CMN1}
Let $G$ be a finite group. The rank of $TF(G)$ is equal to the
number~$n_G$ defined above if $G$ has $p$-rank at most $2$, and is equal
to $n_G+1$ if $G$ has rank at least $3$. 
\end{thm}

We say that a  $kG$-module has {\em trivial Sylow restriction} if its 
restriction to a Sylow $p$-subgroup $S$ of $G$ is isomorphic to the direct 
sum of $k$ with some projective module. Equivalently, a $kG$-module 
with trivial Sylow restriction is the direct sum of a trivial source
endotrivial $kG$-module and some projective module.
In particular, its equivalence class is in the
kernel of the restriction map $T(G) \to T(S)$. The next result is well
known and very important to our development. Its proof is based on the
fact that an indecomposable module with trivial Sylow restriction is 
a direct summand of $k_S^{\uparrow G}$ where $S$ is a Sylow
$p$-subgroup of $G$.  
 
\begin{prop}\label{prop:normal}
If $G$ has a nontrivial normal $p$-subgroup, then every indecomposable
$kG$-module with trivial Sylow restriction has dimension one. 
\end{prop}

Another easy result that we find useful is the following. 

\begin{prop} \label{prop:selfnormal}
Suppose that a Sylow $p$-subgroup $S$ of $G$ is self-normalizing (i.e. 
$N_G(S) = S$). Then the only indecomposable $kG$-module with trivial 
Sylow restriction is the trivial module. 
\end{prop}

\begin{proof}
The Green correspondent of any indecomposable  $kG$-module $M$ with 
trivial Sylow restriction must have dimension one by the above proposition. 
Hence the Green correspondent is the trivial module and  $M \cong k$. 
\end{proof}

The following theorem has several applications to finite groups of Lie type. Note that 
the first condition in the statement is equivalent to saying
that $G$ contains the derived subgroup 
$[H,H] \times [J,J]$ of $H\times J$. 

\begin{thm} \label{thm:directprod}
Suppose that $H$ and $J$ are finite groups and that $G$ 
is a normal subgroup of the direct product $H \times J$ with abelian
factor group $(H\times J)/G$ and such that
the orders of both $G \cap H$ and $G \cap J$ are divisible
by $p$ (here we are identifying $H$ with $H \times \{1\}$
and $J$ with $\{1\} \times J$ in $H \times J$). Then any 
indecomposable $kG$-module with trivial Sylow restriction 
has dimension one. 
\end{thm}

\begin{proof}
Let $\whh = H \cap G$ and $\whj = J \cap G$. 
Let $Q$ and $T$ denote Sylow $p$-subgroups
of $\whh$ and $\whj$, respectively, and let 
$S$ be a Sylow $p$-subgroup of $G$ that contains $Q \times T$. Note that
$T,Q\trianglelefteq S$.
Let $W = (\whh \times \whj)S$. 
Suppose that $M$ is an indecomposable  
$kW$-module with trivial Sylow restriction.
Then $M_{\downarrow \whh S} \cong \chi \oplus \proj$ for 
some indecomposable $k(\whh S)$-module $\chi$. We know
that $\chi$  has dimension one 
because $\whh S$ has a nontrivial normal $p$-subgroup, namely $T$.
Moreover,  $T$ is centralized by every element of $\whh$. 

It follows that $M$ is a direct summand of $\chi^{\uparrow W} 
\cong kW \otimes_{k(\whh S)} \ \chi$. Observe that all 
of the left coset representatives of $\whh S$ in $W$ 
can be taken to be elements of $\whj $. Because these
elements centralize $Q$, it must be that $Q$ acts trivially
on $\chi^{\uparrow W}$ and hence also on $M$. Therefore,  
the restriction of $M$ to $S$ can have no nonzero projective 
summands and $M$ must have dimension one. 

Suppose that $N$ is a $kG$-module with trivial Sylow 
restriction. Then $N_{\downarrow W} \cong \Theta \oplus \proj$,
where $\Theta$ has dimension one. This means that $N$ 
is a direct summand of $\Theta^{\uparrow G}$ and because
$W$ is normal in $G$, $(\Theta^{\uparrow G})_{\downarrow W}$
is a direct sum of conjugates of $\Theta$. It follows that 
$N$ must have dimension one. 
\end{proof}

%%%%%%%%%%%%%%%%%%%%   section 3  %%%%%%%%%%%%%%%%%%%%%%%%%%%%%%%

\section{The Main Method}\label{sec:maintec}
In this section we introduce conditions that imply the triviality
of any indecomposable $kG$-module with trivial Sylow restriction. 
The method was suggested by the work of Balmer \cite{Bal},
though none of the results of \cite{Bal} are directly required in this paper. It
is worth pointing out that the method works for perfect groups 
(i.e., $[G,G] =G$), and, with some effort, it can be adapted to other cases to
prove that indecomposable $kG$-modules with trivial Sylow restriction
have dimension one.
The statement proved in Theorem~\ref{thm:method} below is sufficient for
this paper. A somewhat different version of the method is contained in
the paper \cite{CT4}.  

For each nontrivial $p$-subgroup $Q$ of
a given Sylow $p$-subgroup $S$ of $G$, we
construct a chain of subgroups: 
$$\rho^1(Q) \subseteq \rho^2(Q) \subseteq \dots.$$
These were written $\rho_{i-1}(Q)$ in \cite{CMN3}
where they were first introduced.
The subgroups are defined inductively by the following rule:
\begin{align*}
\rho^1(Q)&=[N_G(Q),N_G(Q)]    \qquad \text{and}\\
\rho^i(Q)&=\langle N_G(Q)\cap\rho^{i-1}(R)~|\{1\} \neq R\subseteq S\rangle
\qbox{for $i>1$.}
\end{align*}
In \cite{CT4}, it is shown that if $\rho^i(S) = N_G(S)$ for some $i$
(or more generally if $\rho^i(Q) = N_G(Q)$ for some nontrivial subgroup
$Q \subseteq S$ with $N_G(S) \subseteq N_G(Q)$), then the trivial
$kG$-module is the only indecomposable module with trivial Sylow
restriction. The following theorem (Theorem~\ref{thm:method}) is the simplified version of that
result needed for most of this  paper.

\begin{thm} \label{thm:method}
Let $S$ be a Sylow $p$-subgroup of $G$, and let $H$
be a subgroup of $G$ such that $N_{G}(S)\leq H$. 
Suppose that the following conditions hold.  
\begin{itemize}
\item[(A)] Every indecomposable $kH$-module with trivial
Sylow restriction has dimension one.
\item[(B)] $H = \langle g_1, \dots, g_m \rangle$ such that for each $i$, either
\begin{itemize}
\item[(1)] $g_i \in [H,H]S$, or
\item[(2)] there exists a subgroup $H_i$ of $G$ such that 
\begin{itemize}
\item[(a)] every indecomposable $kH_i$-module with trivial Sylow 
restriction has dimension one, 
\item[(b)] $p$ divides the order of $H_i \cap H$, and
\item[(c)] $g_i \in [H_i,H_i]$. 
\end{itemize}
\end{itemize}
\end{itemize}
Then the trivial module $k$ is the only indecomposable $kG$-module with
trivial Sylow restriction.
\end{thm}

\begin{proof}
Suppose that $M$ is a $kG$-module with trivial Sylow restriction.
Then $M_{\downarrow H} \cong \chi \oplus \proj$ for some $kH$-module
$\chi$ having dimension one. So $[H,H]$ and $S$ 
are in the kernel of $\chi$ and any generator $g_i$ of $H$ that 
satisfies condition (1) must act trivially on $\chi$. Our next objective
is to prove that the same holds for any generator $g_i$ of $H$
satisfying condition (2).  

Suppose that $g_i$ satisfies condition (2) for some subgroup $H_i$ of
$G$. By (2)(b), we can pick a nontrivial $p$-subgroup 
$Q_i \subseteq H_i\cap H$ for each $i$. By condition (2)(a),
$M_{\downarrow H_i} \cong \mu \oplus \proj$
for some one-dimensional $kH_{i}$-module $\mu$. Since $g_i$
is in $[H_i,H_i]$ by (2)(c), $g_i$ acts trivially on $\mu$. 
As $p$ divides the order of $H_i \cap H$ by (2)(b), any projective
$k(H_i \cap H)$-module has dimension divisible by $p$. So consider the
restriction 
\[
M_{\downarrow (H_i \cap H)}  \cong 
\chi_{\downarrow (H_i \cap H)} \oplus \proj 
\cong  \mu_{\downarrow (H_i \cap H)} \oplus \proj.
\]
By the Krull-Schmidt Theorem 
$\mu_{\downarrow (H_i \cap H)}\cong \chi_{\downarrow (H_i \cap H)}$, and
hence $g_i$ acts trivially on $\chi$. 

Since every generator of $H$ acts trivially on $\chi$, 
it follows that $\chi \cong k_H$, the trivial $kH$-module. 
Now, $M$ is 
indecomposable and $H$ contains the normalizer of $S$. So
$M$ must be the Green correspondent of $k_H$. 
That is, $M \cong k$, as asserted. 
\end{proof}

\begin{rem}
In most of the applications of Theorem \ref{thm:method} in this paper,
the group $G$ is a special linear group and the subgroup $H$ is a
parabolic or Levi subgroup that contains the normalizer of a Sylow
$p$-subgroup of $G$. 
For such subgroups, condition (A) in the hypothesis of the theorem is
established using an argument similar to that of Theorem
\ref{thm:directprod}.

In the case that $H = N_G(Q)$ where 
$Q$ is a nontrivial characteristic subgroup of the Sylow $p$-subgroup
$S$ of $G$, the hypotheses of Theorem~\ref{thm:method} basically say
that $\rho^2(Q)= N_G(Q)$, which guarantees that the trivial $kG$-module
is the only indecomposable $kG$-module with trivial Sylow restriction.
In all but one of the proofs of Sections 
\ref{sec:SLncase} and \ref{sec:thm11},  this
information is sufficient to obtain the asserted result. There is a
unique case for which we need to compute $\rho^3(Q)$, relying on
information gathered in \cite{CMN3}.
\end{rem}

\begin{rem}
It should be pointed out that conditions (B)(1) and (2) on the generators
of $H$ in the hypothesis of the theorem are not inherited by
subgroups. 
That is, if $J$ is subgroup of $H$ also containing the normalizer of a 
Sylow $p$-subgroup of $G$, and $H$ satisfies condition (B)(1) or (2),
then we cannot conclude that $J$ satisfies condition (B)(1) or (2)
respectively.
\end{rem}

%%%%%%%%%%%%%%%%%%%%  section 4  %%%%%%%%%%%%%%%%%%%%%%%%%%%%%%%%%

\section{Groups of Lie Type A}\label{sec:group-structure} 
In this section we recall some known facts on the structure of the Sylow
$p$-subgroups and their normalizers for finite groups of Lie type A in
nondefining characteristic. More information can be found in
\cite{AM,AF,GLS,W}. 

For convenience we set some notation that is used
throughout the rest of paper.

\begin{notation} \label{note:basic}
Let $k$ be a field of prime characteristic $p$ and $q$ a prime power
such that $\gcd(p,q)=1$. Let $e$
denote the least integer such that $p$ divides $q^e-1$ and write
$q^e-1=p^td$, where $\gcd(p,d)=1$ and $t\geq1$. 
Given a positive integer $n$, let
$r,f$ be integers such that $n = re+f$ and $0 \leq f < e$. 
\end{notation}

Thus, $e$ is the multiplicative order of $q$ modulo $p$, and 
$p^t$ is the highest power of $p$ dividing $q^e-1$. 
In particular, $e$ is the smallest integer such that $p$ divides the
order of $\GL(e,q)$. 

We start with the following useful elementary observations.

\begin{prop} \label{prop:sylow1}
Suppose that $G$ is a group such that 
$\SL(n,q) \subseteq G \subseteq \GL(n,q)$ and let $S$ be a Sylow
$p$-subgroup of $G$. 
Let $\Det(G) \subseteq \bfq^\times$ be
the image of the determinant map. 
\begin{itemize}
\item[(a)] $G$ is the subgroup of $\GL(n,q)$ consisting of all
invertible matrices whose determinants are in $\Det(G)$. 
\item[(b)] $S$ is abelian if and only if and only if $n < pe$.
\item[(c)] The $p$-rank of $G$ is $r$ except in the case that 
$p$ divides both $n$ and $q-1$.  In that case, the $p$-rank is 
either $r$ or $r-1$, depending on whether the order of 
$\Det(G)$ is divisible by $p$. 
\end{itemize}
\end{prop}

\begin{proof}
(a) is immediate.
For (b) and (c), see \cite{GLS} or \cite{W}.
\end{proof} 

In general, a Sylow $p$-subgroup $S$ of $G$ is a subgroup of a direct
product of iterated wreath products.  
For $G= \GL(n,q)$, a Sylow $p$-subgroup $S$ of $G$ is the Sylow $p$-subgroup
of a semi-direct product $(C_{p^t}\rtimes C_e)^r \rtimes \fS_r$, where $\fS_r$ is
the symmetric group on $r$-letters (\cite[Theorem~VII.4.1]{AM}). 
For any $\SL(n,q)\subseteq G \subseteq \GL(n,q)$, a Sylow $p$-subgroup of $G$ 
is the intersection of $G$ with a Sylow $p$-subgroup of $\GL(n,q)$. 
Recall that a Sylow $p$-subgroup $R$ of $\fS_r$ is a direct product of
iterated wreath products as follows. Write $r=\sum_{0\leq i\leq M}a_ip^i$ with
$0\leq a_i<p$ for each $i$. Then 
$$R\cong\prod_{0\leq i\leq M}\big(\underbrace{C_{p}\wr C_p\wr\dots\wr
  C_p}_{i\;\hbox{\scriptsize{terms}}}\big)^{a_i}=\prod_{0\leq i\leq M}\big(C_p^{\wr
  i}\big)^{a_i}$$
where $C_p^{\wr i}$ is a Sylow $p$-subgroup of $\fS_{p^i}$.

\begin{thm}\label{thm:p-local}
Suppose that $p>2$. With the above notation, the following hold. 
\begin{itemize}
\item[(a)]
$S\cong\prod_{0\leq i\leq M}\big(C_{p^t}\wr(C_p^{\wr i})\big)^{a_i}$
\item[(b)] Each of the $r$ factors $C_{p^t}$ of $S$ can be embedded as Sylow
$p$-subgroup of a diagonal block $\GL(e,q)$ of $\GL(n,q)$, and the other
  generators of $S$ can be embedded as permutation matrices of these
  blocks according to the $p$-adic expansion of $r$. 
In other words, $S$ can be chosen in a
Levi subgroup of $\GL(n,q)$ with diagonal blocks of size
$$(\underbrace{e,\dots,e}_{a_0\;\hbox{\scriptsize{terms}}},
\underbrace{ep,\dots,ep}_{a_1\;\hbox{\scriptsize{terms}}},\dots,
\underbrace{ep^M,\dots,ep^M}_{a_M\;\hbox{\scriptsize{terms}}},
\underbrace{1,\dots,1}_{f\;\hbox{\scriptsize{terms}}})$$
\item[(c)] The normalizer $N_{\GL(n,q)}(S)$ of $S$ is contained in the
  normalizer of the Levi subgroup containing $S$ above. 
\item[(d)] $S$ contains a unique elementary abelian subgroup $E$ of
rank $r$, hence characteristic in $S$, and each elementary abelian
subgroup of $S$ is conjugate to a subgroup of $E$. 
\end{itemize}
\end{thm}

\begin{proof}
See \cite[Section 4]{AF}, \cite[Section VII]{AM}, \cite[Theorem~4.10.2
  and Remark 4.10.4]{GLS} and \cite[Section 2]{W}.
\end{proof}

The case $p=2$ is handled separately, as the $2$-local structure of
$\GL(n,q)$ and subgroups is very different from the case $p>2$.
 
The lemma below is well-known. We sketch a proof of the lemma 
because it will be used it several times.

\begin{lemma} \label{lem:Levicom}
Suppose that $n = rs$ for positive integers $r$ and $s$, with 
$r > 1$. In $\whg = \GL(n,q)$ let $\whl \cong \GL(s,q)^r$ be 
the Levi subgroup of all elements that can be written as 
block diagonal $s\times s$ matrices. Let $L = \whl \cap G$
where $G = \SL(n,q)$, and let $N = N_G(L)$.

\begin{itemize}
\item[(a)] If $q$ is odd and $r = 2$, then
the quotient $N/[N,N]$ is a Klein four group.
\item[(b)] If $q$ is odd and $r > 2$, or if $q$ is even, then the 
commutator subgroup of $N$ has index $2$ in $N$. 
\end{itemize}
\end{lemma}

\begin{proof}
The subgroup $N$ is an extension
\[
\xymatrix{
1 \ar[r] & L \ar[r] & N \ar[r] & \fS_r \ar[r] & 1
}
\]
where the symmetric group $\fS_r$ acts on $L$ by permuting
the diagonal blocks. We know that $[L,L] \cong \SL(s,q)^r$
and can identify $L/[L,L]$ with the subgroup
of $(\bfq^\times)^r$ given as $L/[L,L] \cong \{(a_1, \dots, a_r)
\in (\bfq^\times)^r \ \vert \ a_1 \cdots a_r = 1\}$. Thus,
$N/[L,L]$ is an extension
\[
\xymatrix{
1 \ar[r] & L/[L,L] \ar[r] & N/[L,L] \ar[r] & \fS_r \ar[r] & 1,
}
\]
where the symmetric group acts by permuting the places.

If $r = 2$, then $N/[L,L]$ is a dihedral group of order $2(q-1)$ whose 
commutator subgroup is cyclic of index $4$ if $q$ is odd, and of index
$2$ if $q$ is even. Therefore, the quotient group
$N/[N,N]$ is a Klein four group if $q$ is odd, respectively cyclic of
order $2$ if $q$ is even.

Now assume that $r > 2$. It is easy to see
that $L/[L,L]$ is generated by the element $\alpha = (a, a^{-1},
1,\dots, 1)$ and its conjugates under the action of the symmetric group,
where $a$ is a generator for $\bfq^\times$. 
One of these conjugates is $\beta = (1, a, a^{-1}, 1, \dots, 1)$. For
$\sigma = (1,2) \in \fS_r$ we calculate
$[\alpha\beta,\sigma]=\alpha\beta\sigma(\alpha\beta)^{-1}\sigma^{-1} =
\alpha$. Hence, $\alpha$ and all of its conjugates under the action of the 
symmetric group are in the commutator subgroup of $N/[L,L]$ 
and hence $L \subseteq [N,N]$. On the other hand, the quotient group
$N/L$ is isomorphic to $\fS_r$ and as $N/[N,N]$ is the largest abelian
quotient of $N/L$, and $N/[N,N]$ must have order $2$.
\end{proof}

In the specific context of the section, Theorem~\ref{thm:directprod}
leads to the following observation.

\begin{prop}\label{prop:directprod}
Assume that Notation~\ref{note:basic} holds.
Let $n = n_1 + n_2 + \dots + n_m$, where $n_1, \dots, n_m$ 
are positive integers and let 
\[ 
\whl \quad = \quad \prod_{i = 1}^m \GL(n_i, q) \quad \subseteq
\quad \GL(n,q)
\]
be the Levi subgroup of diagonal blocks of sizes $n_1, \dots,
n_m$. Let $L = \SL(n,q) \cap \whl$. Assume further that 
\begin{itemize}
\item[(a)] if $p$ divides $q-1$ then at least two of $n_1, \dots, n_m$
are greater than one,
\item[(b)]  if $e$ is the smallest positive integer such that $p$
divides $q^e-1$ and $e >1$, then at least two of $n_1, \dots,
n_m$ are greater than or equal to $e$. 
\end{itemize}
Then any indecomposable $kL$-module with trivial Sylow restriction
has dimension one. 
\end{prop}

\begin{proof} Express $\{1, \dots, m\} = A \cup B$ 
as a union of disjoint subsets
such that in case (a) each of $A$ and $B$ contains some index $i$
such that $n_i > 1$, or in case (b), each of $A$ and $B$ contains
an index $i$ such that $n_i \geq e$. Then let 
$H = \prod_{i \in A} \GL(n_i,q)$, $J = \prod_{i \in B} \GL(n_i,q)$.
Then $\whl \cong H \times J$, and Theorem \ref{thm:directprod} proves
the assertion for $L = \SL(n,q) \cap \whl$. Indeed, $\whl/L$ is abelian
and the conditions (a) and (b) ensure that the orders of 
$H \cap L$ and $J \cap L$ are both divisible by $p$.
\end{proof}

We end the section by recalling the following result (cf.  
\cite[Theorem~3.4]{CMN3}).  

\begin{thm} \label{thm:torfree}
Let $G$ be a group such that $\SL(n,q) \subseteq G
\subseteq\GL(n,q)$. Suppose that a Sylow $p$-subgroup of
$G$ has $p$-rank at least $2$. Then $TF(G) \cong \bZ$.
\end{thm}

Note that the theorem excludes the groups $\SL(2,q)$ for $p=2$, in which
case a Sylow $2$-subgroup is generalized quaternion. 

\begin{cor} \label{cor:torfree}
Let $G$ be a group such that $\SL(n,q) \subseteq G
\subseteq\GL(n,q)$. Suppose that $Z \subseteq Z(G)$ and 
that $G$, $Z$ satisfy the conditions of Theorem \ref{thm:noncyclic1a}. 
Then $TF(G/Z) \cong \bZ$.
\end{cor}

\begin{proof}
Let $T$ be the subgroup of all elements of order $p$ in 
the torus of diagonal $e \times e$ block matrices in $\GL(n,q)$. 
The point of the proof of Theorem \ref{thm:torfree} is that every 
elementary abelian $p$-subgroup of $G$ is conjugate to a subgroup
of $G \cap T$. From this is follows that if $G$ has maximal elementary
abelian subgroups of rank $2$, then they are all conjugate to a subgroup
of $G\cap T$ and the 
conclusion follows from Theorem \ref{thm:rank}. This is also
true for $G/Z$ if $T \cap Z$ is trivial. 

Consequently, the only remaining 
cases occur when $T \cap Z$ is not trivial. This requires that 
$p$ divide $q-1$ or equivalently that $e=1$. Now $T\cap Z$ is a cyclic
central subgroup of $G$, so that it is still the case that 
every elementary abelian $p$-subgroup is conjugate to one generated
by elements that are the classes modulo $Z$ of diagonal matrices. 
If $n \geq 4$ then every maximal elementary abelian $p$-subgroup 
has rank at least $3$ and again we are done. The same happens if 
$n = 3$ and either $p>3$ or if $n=p=3$ and $Z$ does not contain the Sylow
$p$-subgroup of $Z(G)$. This proves the corollary.
\end{proof}

%%%%%%%%%%%%%%%%%%%%%%%%%%%%%%%%  section 5 %%%%%%%%%%%%%%%%%%%%%%%

\section{Endotrivial Modules for $\SL(n,q)$}\label{sec:SLncase}
The aim of this section is to prove Theorem~\ref{thm:noncyclic1a}
in the case that $G = \SL(n,q)$ and that $Z$ is trivial. 
Throughout this section  we assume Notation~\ref{note:basic}. Thus, 
$q^e-1=p^td$ where $d,e$ and $t$ are positive integers
such that $p$ does not divide $d$, and $e$ is the multiplicative order
of $q$ in the base field $\bF_p\subseteq k$. 
The assumption that the Sylow $p$-subgroup of 
$G$ is not cyclic is equivalent to the condition that $n \geq 2e$. 

The proof is split into several cases. The first case is when $e$ 
divides $n$ but the quotient $n/e$ is not a power of $p$. 

\begin{prop} \label{prop:composite}
Suppose that $G = \SL(re, q)$ for $r\geq 2$ not a power of $p$.
Assume also that if $p=2$, then $r \geq 4$. 
Then the trivial module $k$ is the only indecomposable $kG$-module with
trivial Sylow restriction. In particular, $TT(G) = \{0\}$.
\end{prop} 

\begin{proof}  First notice that if  $r < p$, then a Sylow 
$p$-subgroup is abelian and the proposition is proved in \cite{CMN3}.
Hence, we may assume further that $n = re > pe$. 

The proof is divided into three cases: 
\begin{itemize}
\item[(i)] $r = 2p^s$ for some $s\geq1$ and $p>2$, 
\item[(ii)] $r = ap^s$ for $2 < a < p$, and 
\item[(iii)] $r = ap^s +b$ for $1 \leq a < p$ and $1 \leq b < p^s$.
\end{itemize}

Note that in cases (i) and (ii) we may assume that $e>1$ and that $p>2$,
as otherwise, $p$ divides both $n$ and $q-1$. In that case,  $\SL(n,q)$ is a
perfect group with a nontrivial normal $p$-subgroup,
and the proposition is a consequence of Proposition~\ref{prop:normal}.

In the first two cases, 
let $m = p^se$, so that $n = am$. Let
\[
\whl \quad = \quad \whl(m, \dots, m) 
\quad \cong \quad \GL(m,q)^a \quad \subseteq \quad
\GL(n,q)
\]
be the Levi subgroup consisting of $a$ diagonal $m \times m$
blocks. Let $L = \whl \cap G$ and $N = N_G(L)$. The group 
$N$ is an extension (perhaps not 
split) of the form 
\[
\xymatrix{ 
0 \ar[r] & L \ar[r] &  N \ar[r] & \fS_a \ar[r] & 0,
}
\]
where $\fS_a$ is the symmetric group on $a$ letters. In addition, 
$N$ contains the normalizer of a Sylow $p$-subgroup of $G$
(cf. Lemma~\ref{thm:p-local}). 

\vskip.05in
\noindent{\it Case} (i).
Suppose that $a=2$, and $n = am = 2p^se$. 
The commutator subgroup $[N,N]$ must contain the perfect group
$\SL(m,q)\times\SL(m,q)$. By Lemma~\ref{lem:Levicom}, if $q$ is odd,
then the quotient group $N/[N,N]$ is a Klein four group, and we see that we can choose generators 
represented by the elements 
\[
\sigma  =  \begin{bmatrix}  & I_m \\ -I_m  \end{bmatrix}
\qquad \hbox{and} \qquad 
\tau  =  \begin{bmatrix} c   \\
		& I_{m-1}  \\
		 &  & c^{-1}  \\
		 &  &  & I_{m-1} \end{bmatrix}
\]
where $c$ is a generator for the Sylow $2$-subgroup of
$\bfq^\times$. If $q$ is even, then $N/[N,N]$ has order $2$ and is
generated by $\sigma$ (remembering that $-1 = 1$). 
To invoke Theorem \ref{thm:method} is is enough to show that $\sigma$ and
$\tau$, are in the commutator subgroup of
the normalizer of some nontrivial $p$-subgroup of $N$. 

There is an embedding $\varphi: \bF_{q^e} \to \hbox{Mat}_e(\bfq)$
where Mat${}_e(\bfq)$ is the algebra of $e \times e$ matrices
over $\bfq$. This is given as the action of the algebra $\bF_{q^e}$
on itself, but regarded as a vector space over $\bfq$. From 
this we get a homomorphism $\hat{\varphi}: \GL(2p^s,q^e) 
\to \GL(2p^se, q)$. That is, the map $\hat{\varphi}$ replaces an element 
given by a matrix $(a_{i,j})$ by the block matrix $(\varphi(a_{i,j}))$.
Again, $\hat{\varphi}$ can be obtained by taking the natural module
for $\GL(2p^s,q^e)$ and writing it as a module over $\bfq$ of dimension
$2p^se$.

The group $\SL(2p^s,q^e)$ has a central element $Y$ of order $p$ since
$s>0$.  Observe that $\hat{\varphi}(Y)$ is also in $N$. 
Let $H_1 = C_G(\hat{\varphi}(Y))$, which contains the image
$\hat{\varphi}(\SL(2p^s,q^e))$. In particular, we have that
$\varphi(-1) = -I_e$, and so for
\[
X = \begin{bmatrix} & I_{p^s} \\ -I_{p^s} &  \end{bmatrix}
\qquad \text{then} \qquad \hat{\varphi}(X) 
= 
\begin{bmatrix}  & I_{m} \\ -I_{m} &  \end{bmatrix}
=  \sigma.
\]
Note that $X$ is in $\SL(2p^s,q^e)$, and hence $\sigma$ is
in the commutator subgroup of $H_1$. Moreover, 
because $H_1$ has a central element of order $p$, any 
indecomposable $kH_1$-module 
with trivial Sylow restriction has dimension one. Thus $H_1$ and 
$g_1 =\sigma$
satisfy condition (B)(2) of Theorem \ref{thm:method} with $H = N$.
Clearly, any element of $[N,N]$, satisfies condition (B)(1) of Theorem
\ref{thm:method}, which implies that the proposition holds in case (i)
if $q$ is even, because $N=\langle [N,N],g_1\rangle$. 

To finish the proof for $q$ odd, we prove the similar result for $\tau$. 
Let $\whh_2 = \whl(2m-e,e) \subseteq \GL(n,q)$ be the Levi subgroup consisting
of diagonal block matrices of sizes $2m-e$ and $e$, and
let $H_2 = \whh_2 \cap G$. By Proposition
\ref{prop:directprod}, any indecomposable $kH_2$-module with 
trivial Sylow restriction has dimension one. Clearly, $H_2 \cap N$
has order divisible by $p$. The commutator subgroup 
$[H_2,H_2] \cong \SL(2m-e,q) \times \SL(e,q)$ contains the element 
$\tau$. So condition (B)(2) of Theorem \ref{thm:method} is satisfied for
$g_2=\tau$ and $H_2$, and the proposition holds in case (i). 
\vskip.05in

\noindent{\it Case} (ii).  Now suppose that $2 < a < p$.
In this case, the quotient group $N/[N,N]$ has 
order $2$ and a generator is represented by the element 
\[
\sigma =  \begin{bmatrix} & I_m &  \\ 
		-I_m & \\
		& & I_{(a-2)m}  \end{bmatrix}
\]
Again, it is enough to show that $\sigma$, is in
the commutator subgroup of an appropriate subgroup of $G$ to invoke
Theorem \ref{thm:method}.  Let 
$\whl = \whl(2m,(a-2)m)\cong \GL(2m,q) \times \GL((a-2)m,q)$
be the Levi subgroup of diagonal block matrices of size $2m$ and
$(a-2)m$, for $m =p^se$.
Let $H_1 =\whl \cap G$. Every indecomposable $kH_1$-module with
trivial Sylow restriction has dimension one, by Proposition
\ref{prop:directprod}.
Clearly, $H_1 \cap N$ has order divisible by $p$, and $\sigma$,
is in $[H_1,H_1] \cong \SL(2m,q) \times
\SL((a-2)m,q)$. Again Condition (B)(2) of Theorem \ref{thm:method} holds
for $\sigma$, and $H_1$.
So the proposition is proved also in case (ii). 
\vskip.05in
\noindent{\it Case} (iii). Let 
$\whl = \whl(ap^se,be)\cong \GL(ap^se,q) \times \GL(be,q)$ 
be the Levi subgroup of blocks of size $ap^se$ and $be$, and put
$N = \whl \cap G$. 
Observe that,  $N$ contains the 
normalizer of a Sylow $p$-subgroup
of $G$. Thus by Proposition \ref{prop:directprod}, 
any indecomposable $kN$-module with trivial 
Sylow restriction has dimension one.

The commutator subgroup of $N$ is the direct product
$\SL(ap^se,q)\times\SL(be,q)$, 
implying  that $N/[N,N] \cong \bfq^\times$. Hence, 
$N$ is generated by $[N,N]$ and a diagonal matrix $\sigma$ with diagonal
entries $1, 1, \dots, 1, w, w^{-1},1, \dots, 1$ where $w$ is a generator
of $\bfq^\times$ and the nonidentity entries occur in rows  $ap^se$ and
$ap^se+1$.  

Now let $\whh_1 = \whl(ap^se-1,be+1)
\cong \GL(ap^se-1,q) \times \GL(be+1,q)$, the Levi subgroup of 
blocks of size $ap^se-1$ and $be+1$. Let $H_1= \whh_1 \cap G$.
It is straightforward to show that condition (B)(2) of Theorem
\ref{thm:method} is satisfied for $g_1=\sigma$, $H_1$ and $H = N$, and
the proposition holds in case (iii). This completes the proof. 
\end{proof}

The  next step is the following. 

\begin{prop} \label{prop:sleps}
Suppose that $G = \SL(p^se, q)$ and $s \geq 1$. Then 
any indecomposable $kG$-module with trivial Sylow restriction 
has dimension one. Thus, if $p^se > 2$, then $TT(G) = \{0\}$.  
\end{prop}

\begin{proof}
First we should notice that if $e = 1$, that is, if $p$ divides $q-1$, 
then $G$ has a central subgroup of order $p$, and we are done 
by Proposition \ref{prop:normal}. So assume that 
$e>1$. This assumption requires that $p > 2$. 

Let $\theta: \GL(e,q)^{p^s} \to \GL(ep^s,q)$ be the injective group
homomorphism given by letting 
$\theta(A_1, \dots, A_{p^s})$ be the block diagonal matrix of 
$e \times e$ blocks $A_1, \dots, A_{p^s}$:
\[
\theta(A_1, \dots, A_{p^s}) =  
\begin{bmatrix}
A_1 &&&& \\
& A_2 &&& \\
&& \dots && \\
 &&&&	A_{p^s}
\end{bmatrix}
\]
Choose an element $u \in \SL(e,q)$ of order $p$. 
For $i = 1, \dots, p^s$, let $x_i = \theta(A_1, \dots, A_{p^s})$
where $A_i = u$  and $A_j = I_e$, the $e \times e$ identity matrix,
for $j \neq i$.  Let $Q = \langle x_1, \dots, x_{p^s} \rangle$. 

Since $p$ is odd, $Q$ is the unique elementary abelian
subgroup of rank $p^s$ in some given Sylow $p$-subgroup $S$ of $G$ and
each elementary abelian subgroup of $S$ is conjugate to a subgroup of
$Q$, by Theorem~\ref{thm:p-local}.
Thus $Q$ is characteristic in $S$, which implies that 
$N_G(S)\subseteq N_G(Q)$. Hence, we may apply Theorem \ref{thm:method} to
$H = N_G(Q)$.

Write $S=C\wr R$ where $u\in C$ and $C\cong C_{p^t}$ is a Sylow
$p$-subgroup of $\GL(e,q)$, and where $R\cong (C_p)^{\wr s}$ is a Sylow
$p$-subgroup of $\fS_{p^s}$. Note that $\langle u\rangle\subseteq C$
with equality if and only if $t=1$.

From \cite[Section 6]{CMN3}, we have 
$N_{\GL(e,q)}(\langle u\rangle)=N_{\GL(e,q)}(C)=\langle w,g\rangle\cong
C_{q^e-1}\rtimes C_e$ where $\ls gw=w^q$. 
Hence, $H = N_G(Q)$ is an extension
$$
\xymatrix{
1\ar[r]&  J \ar[r]  &  H \ar[r]   &  \fS_{p^s}\ar[r] & 1
}
$$
where 
\[
J = N_{\GL(e,q)}(C)^{p^s}\cap G =
\{\theta(A_1,\dots,A_{p^s})~|~A_i\in N_{\GL(e,q)(C)}~,~
\prod_{1\leq  i\leq p^s}\Det(A_i)=1\}
\] 
Thus $J$ is generated by conjugates under $\fS_{p^s}$ of elements of the form
\[
a = \theta(A_1,A_2,I_e,\dots,I_e)\qbox{where}
\]
$A_1,A_2\in N_{\GL(e,q)}(C)$ and $\Det(A_1)\Det(A_2)=1$.
So $H$ is generated by $J$ and elements of the form
\[
X_i = \begin{bmatrix}I_{e(i-1)}&&\\&\tau&\\&&I_{e(p^s-i-1)}\end{bmatrix}
\quad\hbox{with}\quad
\tau=\begin{bmatrix}&I_e\\-I_e&\end{bmatrix}
\]
for $1\leq i\leq p^s-1$. Note that all $X_i$ are conjugate. 

For $i = 1, \dots, p^s-1$, let 
$R_i=\langle x_i,x_{i+1}\rangle\subset S$. Then $N_G(R_i)$ is an
extension
\[
\xymatrix{
1\ar[r]&(N_{\GL(e,q)}(C)\wr\fS_2)\cap  G\ar[r]
&N_G(R_i)\ar[r]&\SL(e(p^s-2),q)\ar[r]&1
} 
\]

If $p^s > 3$, then the elements $a$ and $X_1$ lie in the commutator
subgroup of $N_G(R_3)$. A similar condition holds for any conjugates of
$a$ and $X_1$ under the action of $\fS_{p^s}$. By applying Theorem
\ref{thm:method} to $H = N_G(Q)$ and the generators given above we are
done.  

We are left with the case $p^s = 3$ and $e = 2$.
A computer calculation shows that for $G = \SL(6,2)$ with $p = 3$,
we have $\rho^2(Q) \neq N_G(Q)$. Hence,  the method of Theorem
\ref{thm:method} fails. One the other hand, \cite[Proposition 7.9]{CMN3} 
shows that $a,X_1\in \rho^2(R_1)$ and
therefore all their conjugates are in $\rho^3(Q)$. Thus
$N_G(Q)=\rho^3(Q)$, and Corollary 4.6 of \cite{CT4} asserts
that  $TT(G)=\{0\}$ in this case.
\end{proof}

We are now ready for the proof of the main theorem of the section. 

\begin{thm} \label{thm:slgencase}
Assume Notation~\ref{note:basic}. 
Let $G = \SL(n, q)$ with $n \geq 2e$ if $p$ is odd, or $n\geq3$
if $p =2$. Then the trivial $kG$-module is the unique indecomposable
$kG$-module with trivial Sylow restriction.
\end{thm}

\begin{proof}
Let $n=re+f$ with $r\geq 2$ and $0 \leq f < e$.  
By Propositions \ref{prop:composite} and \ref{prop:sleps}, the
theorem is true if $f = 0$.  
Hence, we assume that $f > 0$ and thus also $e >1$. 
There is a  natural embedding of $\SL(n-1,q)\hookrightarrow
\SL(n,q)$. It is an easy exercise to show that the index of 
$\SL(n-1,q)$ in $\SL(n,q)$ is prime to $p$ and,  hence, 
$\SL(n-1,q)$ contains a Sylow $p$-subgroup of $\SL(n,q).$
By \cite[Theorem 9.6]{CMN3}, the restriction map $T(\SL(n,q))
\to T(\SL(n-1,q))$ is injective since $e>f\geq 1$.
Therefore, by induction on $f$, the proof of the theorem is
complete.
\end{proof}

%%%%%%%%%%%%%%%%%%%%%   section 6  %%%%%%%%%%%%%%%%%%%%%%%%%%%%%%

\section{Proof of Theorem~\ref{thm:noncyclic1a}}\label{sec:thm11}
The proof of Theorem~\ref{thm:noncyclic1a} is a consequence 
of Theorem \ref{thm:slgencase} and a case by case inspection depending
on the $p$-part of the central subgroup $Z$ of $G$ in
Theorem~\ref{thm:noncyclic1a}.
The next proposition is an essential step in the general proof. 

\begin{prop}\label{prop:rp-case}
Let $G = \SL(n,q)$ where $n = rp \geq 3$ for some $r\geq 1$ and 
assume that $p$ divides
$q-1$. If $n = p = 3$, assume further that
$9$ divides $q-1$. Let $Z$ be a nontrivial central subgroup of $G$.
Then the trivial $k(G/Z)$-module is the unique indecomposable
$k(G/Z)$-module with trivial Sylow restriction.
\end{prop}

\begin{proof}
Recall that, in general, if $A, B, C$ are groups such that $C \subseteq
B \subseteq A$ and $C$ is normal in $A$, then $N_{A/C}(B/C) = N_A(B)/C$.
This fact is used to identify normalizers.

First, we consider the case that $Z$ is a nontrivial $p$-group.
Let $T$ be the torus of diagonal matrices in $G$, and let $Q$ be a Sylow
$p$-subgroup of $T$. We choose $S$ to be a Sylow $p$-subgroup of $G$
that contains $Q$.
We note that $Q$ is characteristic in $S$, it being the
unique abelian subgroup isomorphic to $(C_{p^t})^{n-1}$, where $t$ is the
highest power of $p$ that divides $q-1$. Therefore,
$N_G(S) \ \subseteq \ N_G(Q)$ and we have an extension
\[
\xymatrix{
1 \ar[r] & T \ar[r] & N_G(Q) \ar[r] & \fS_n \ar[r] & 1.
}
\]
There is an inclusion
\[
N_{G/Z}(S/Z) \ = N_G(S)/Z \ \subseteq N_G(Q)/Z \ = N_{G/Z}(Q/Z)
\]
where $N_{G/Z}(Q/Z)$ is an extension
\[
\xymatrix{
1 \ar[r] & T/Z \ar[r] & N_{G/Z}(Q/Z) \ar[r] & \fS_n \ar[r] & 1.
}
\]
Let $N = N_{G}(Q)$. Then $[N, N]$ has index $2$ in $N$, and so
$(N/Z)/[N/Z,N/Z]$ has order two and is generated by the class of the
element 
\[
X \ = \ \begin{bmatrix} U  \\ & I_{n-2} \end{bmatrix}, \quad
\text{where} \quad U \ = \ \begin{bmatrix}  & 1 \\ -1 \end{bmatrix}.
\]

By Theorem \ref{thm:method}, the proof of the theorem is complete in
this case if we show that the class of $X$ is in the commutator subgroup
of some nontrivial $p$-subgroup $R/Z$ of $Q/Z$.  For this let $R$ be the
subgroup generated by
\[
Y  \ = \ \begin{bmatrix} V &  \\  & I_{n-3} \end{bmatrix}, \quad
\text{where} \quad V \ = \ \begin{bmatrix} \zeta &  &
\\  & \zeta \\ & & \zeta^{-2} \end{bmatrix},
\]
where $\zeta$ is a generator for the Sylow $p$-subgroup of 
$\bfq^{\times}.$  Note that the matrix $V$ is not a scalar
matrix. Then the normalizer of $R$ contains the Levi subgroup
\[
L \ = \ (\GL(2,q) \times \GL(1, q) \times \GL(n-3,q)) \cap G.
\]
It follows that the class of $X$ in $G/Z$ is in $[L/Z, L/Z]$ which
is contained in the commutator subgroup
$[N_{G/Z}(R/Z), N_{G/Z}(R/Z)]$. By Theorem
\ref{thm:method}, the trivial $k(G/Z)$-module is the unique
indecomposable module with trivial Sylow restriction.

Next assume that $Z$ is an arbitrary nontrivial central 
subgroup of $G$. Let $\widehat{Z}$
denote the Sylow $p$-subgroup of $Z$. If $M$ is an
indecomposable $k(G/Z)$-module with
trivial Sylow restriction, then $M$ inflates to an indecomposable
trivial Sylow restriction $k(G/\widehat{Z})$-module
$\Inf_{G/Z}^{G/\widehat Z}M$ on which $Z/\widehat{Z}$ acts trivially.
But we have just shown that $\Inf_{G/Z}^{G/\widehat Z}M$ must have
dimension one. Thus $M$ is trivial.
\end{proof}

The following is also required. 

\begin{prop} \label{prop:gensl}
Suppose that $n \geq 2$, $p$ divides $q-1$, 
and if $p = n = 3$ assume that 9 divides 
$q-1$.  Let $S$ be a Sylow $p$-subgroup of $G= \SL(n,q)$ that 
contains the torus $T$ of diagonal matrices of order $p$. Then 
$N_G(T)$ is generated by a collection of elements, each of which is in 
the commutator subgroup of the normalizer of a 
subgroup of $T$ that is not central in $G$.
\end{prop}

\begin{proof}
There is no loss of generality in assuming that $S$ contains the 
Sylow $p$-subgroup of the torus of diagonal matrices of determinant
one, and so contains $T$. 
Let $Y \in S$ be as in the proof of Proposition~\ref{prop:rp-case}. 
Then the commutator subgroup of the normalizer $N_G(\langle Y\rangle)$
of the subgroup generated by $Y$ contains any element of the form
\[
X \ = \ \begin{bmatrix} U & \\  & I_{n-2} \end{bmatrix}, \quad
\text{for} \qquad U \in \SL(2,q).
\]
Any conjugate of $X$, under a permutation matrix $P$, is contained
in the commutator subgroup of the normalizer of $PYP^{-1} \in S$.
It is not difficult to show that $N_G(T)$ is generated by elements
of this form. 
\end{proof}

We can now prove the main theorem. 

\begin{proof}[Proof of Theorem~\ref{thm:noncyclic1a}] 
Assume the notation of Theorem~\ref{thm:noncyclic1a}. 
Note that if $e=1$, $n=2$, $p>2$, and $Z$ does not contain a
Sylow $p$-subgroup of $Z(G)$, then a Sylow $p$-subgroup of 
$G$ is abelian of $p$-rank $2$, and $G/Z$ has a nontrivial normal
$p$-subgroup. 
Thus Theorem \ref{thm:noncyclic1a} holds by Corollary \ref{cor:torfree}
and Proposition \ref{prop:normal}.
Likewise when $n = p = 3$ divides $q-1$ and $Z$ does not contain a 
Sylow $3$-subgroup of $Z(G)$, then $G/Z$ has a nontrivial normal
$3$-subgroup and the conclusion of the theorem follows. If $3$ does not
divide $q-1$, then $n < 2e$ and the theorem does not apply.
In the rest of the proof, we assume that $n\geq 3$ and that if 
$n= 3$ then $p > 3$. 
 
Let $\whg = Z \cdot \SL(n,q)$ and $\whz = Z \cap \SL(n,q)$. Then 
$\whg/Z \cong \SL(n,q)/\whz$. We first prove the theorem
for $G = \whg$. There are two cases to consider. 

Assume first that $p$ does not divide the order of $\whz$. Then 
any indecomposable $k(\whg/Z)$-module with trivial Sylow restriction
inflates to a $k\SL(n,q)$-module with trivial Sylow restriction. By
Theorem \ref{thm:slgencase}, this must be the trivial module. 

Suppose that $p$ divides the order of $\whz$. Because any element
of $\whz$ is a scalar matrix, $p$ must divide $q-1$ and $n$. By 
hypothesis, if $p=2$, then $n >3$; while if $p=3$, then $n > 3$. 
In all cases $n \geq p$. Hence, by Proposition
\ref{prop:rp-case}, the trivial module is the unique indecomposable
$\whg/Z$-module with trivial Sylow restriction.

Next suppose that the index of $\whg$ in $G$ is a power of $p$,
so that $(G/Z)/(\whg/Z)$ is a $p$-group. In this case, $p$ 
divides $q-1$.  For convenience, let 
$K = G/Z$ and $J = \whg/Z$ so that $K/J$ is a $p$-group. Let 
$S$ be a Sylow $p$-subgroup of $K$ and $S^\prime = J \cap S$,
a Sylow $p$-subgroup of $J$. Recall that $ J \cong \SL(n,q)/\whz$.
We may assume that $S^\prime$ contains the image $T$ (modulo $Z$) of 
the torus of diagonal matrices of order $p$ in $\SL(n,q)$, and 
that $T$ is normal in $S$ and $S^\prime$.  
Thus, Proposition \ref{prop:gensl} says that $N_J(T)$ is generated 
by a collection of elements, each of which is in the commutator
subgroup of the normalizer of some 
nontrivial subgroup of $T$, which is not central in $G$, and therefore
cannot be contained in $Z$. 
By Theorem \ref{thm:method}, with $H = N_K(T) = SN_J(T)$, we conclude
that the trivial module is the unique indecomposable $kK$-module with
trivial Sylow restriction. 

Finally, suppose that there is a subgroup $H$ such that $\whg \subseteq
H \subseteq G$, and such that $H/\whg$ is a Sylow $p$-subgroup of
$G/\whg$. Note that by hypothesis, $H/\whg$ is nontrivial.
Theorem~\ref{thm:slgencase} and Proposition~\ref{prop:rp-case} show that
the trivial module is the unique indecomposable $k(H/Z)$-module with trivial
Sylow restriction. Note that the index of $H$ in $G$ is prime to $p$. Suppose
that $M$ is an $k(G/Z)$-module with trivial Sylow restriction. Then
\[ 
M_{\downarrow H/Z} \quad \cong \quad k \oplus \proj,
\]
implying that $M$ is a direct summand of $(k_{H/Z})^{\uparrow G/Z}$.
However, the restriction of $(k_{H/Z})^{\uparrow G/Z}$ to $H/Z$
is a direct sum of copies of $k$, since $H/Z$ is normal 
in $G/Z$ and has index coprime to $p$. 
Both conditions can only occur if $M$ has dimension one. 

We have shown that if $S$ is a Sylow $p$-subgroup of $G/Z$, then 
the kernel of the restriction map $T(G/Z) \to T(S)$ is $X(G/Z)$ the group of 
one-dimensional $k(G/Z)$-modules. The proof of Theorem 
\ref{thm:noncyclic1a} is completed using Corollary \ref{cor:torfree}.
\end{proof}

%%%%%%%%%%%%%%%%%%%%   section 7  %%%%%%%%%%%%%%%%%%%%%%%%%%%%%

\section{Type $A_1$ in characteristic $p \geq 3$} 
\label{sec:gl2}
In the case that $n=2$ and $p$ is odd, 
the Sylow $p$-subgroup of a subquotient of
$\GL(2,q)$ can be cyclic, and so the structure 
of $(G/Z)$ changes accordingly. In this
section, we briefly discuss some cases that were not included
in the results  of \cite{CMN3} and are also excluded from
Theorem~\ref{thm:noncyclic1a} by condition ($b$)
of the hypothesis. The techniques are well known, so only a
sketch of the proof is given. As before, write
$\Det(H)$ for the image under the determinant map
of a subgroup $H$ of $G$ . 

\begin{thm} \label{thm:sl2case}
Assume that $p>2$ and that $p$ divides $q-1$. Suppose that 
$G$ is a group such that 
$\SL(2,q) \subseteq G \subseteq \GL(2,q)$. 
Let $Z\subseteq Z(G)$ be a central subgroup 
of $G$.  Then $|Z|$ divides $2\cdot |\!\Det(G)|$. 
\begin{itemize}
\item[(a)] If $p$ divides $|Z(G):Z|$, that is, 
if $Z$ does not contain the Sylow
  $p$-subgroup of $Z(G)$, then $T(G/Z) = X(G/Z) \oplus \bZ$. 
\item[(b)] Otherwise, $T(G)$ is an extension 
\[
\xymatrix{
0 \ar[r] & X(N_{G/Z}(S)) \ar[r] & T(G) \ar[r] & \bZ/2\bZ \ar[r] & 0
}
\]
\end{itemize}
where $S$ is a Sylow $p$-subgroup of $G/Z$ and right-hand map in 
the sequence is the restriction onto $T(S) \cong \bZ/2\bZ$.
\end{thm}

\begin{proof}
The subgroup $Z$ consists of scalar matrices. 
Let $I_2$ denote the identity matrix in $G$.
If $aI_2\in Z$, then
$a^2\in\Det(G)$. 
It follows that $|Z|$ divides $2|\!\Det(G)|$ as asserted.

A Sylow $p$-subgroup $S$ of $G/Z$ is cyclic if and only if $Z$ contains
the Sylow $p$-subgroup of $Z(G)$. In this case $S$ is isomorphic to a
Sylow $p$-subgroup of $\SL(2,q)$ and $S$ is a TI subgroup of $G/Z$,
implying that the stable categories of $G/Z$ and $N_{G/Z}(S)$ are equivalent.
Part (b) of the theorem follows from \cite[Theorems~3.2 and~3.6]{MT}. 

Otherwise, i.e. if $Z$ does not contains the Sylow $p$-subgroup of
$Z(G)$, then $S$ is not cyclic and Theorem \ref{thm:noncyclic1a}
applies, proving part (a) of the theorem.
\end{proof}

%%%%%%%%%%%%%%%%%%%%%%%   PSL(n, q), n=3 dividing  q-1 %%%%%%%%%%%%%%%

\section{Type $A_2$ in Characteristic 3}\label{sec:char3}

In this section we consider the endotrivial modules for the groups
excluded by condition (c) of the hypothesis of Theorem 
\ref{thm:noncyclic1a}.  Throughout the section we assume the 
following.  Let $n= p =3$, and let 
$q$ be a prime power such that $3$ divides $q-1$
(i.e., $e=1$). Let $\SL(3,q)\subseteq G\subseteq\GL(3,q)$ and $Z$ a
central subgroup of $G$ containing the Sylow $3$-subgroup of the center
$Z(G)$ of $G$.

Note that if either $n=p=3$ does not divide $q-1$ or if $n=2<3=p$ and
$Z$ contains the Sylow $3$-subgroup of the center $Z(G)$ of $G$, then a
Sylow $3$-subgroup of $G/Z$ is cyclic and therefore $T(G/Z)$ is known by
Theorems \ref{thm:cyclic} and \ref{thm:sl2case}. 

\begin{lemma} \label{lem:3structure1}
The group $G/Z$ decomposes as a direct product $G/Z \cong H \times V$
where $Z\cdot \SL(3,q)/Z \subset H$ has index a power of $3$ in $H$ and
$3$ does not divide the order of $V$. 
In particular, $T(G/Z) \cong T(H) \oplus X(G/Z)$, 
where $X(G/Z) \cong X(V)$ is the group of one-dimensional $kV$-modules.
\end{lemma}

\begin{proof}
Since $\Det(G)\subseteq\bF_q^\times$ is an abelian group, we can
write $\Det(G) = U^\prime \times V^\prime$ and 
$\Det(Z) =U^{\prime\prime} \times V^{\prime\prime}$ 
where $U^\prime$, $U^{\prime\prime}$ are $3$-groups and 
$V^\prime$ and $V^{\prime\prime}$ are $3^\prime$-groups.
Let $V = V^\prime/V^{\prime\prime}$. Since $U''\subseteq U'$, we have 
$V \cong \Det(G)/(U'\cdot \Det(Z))$.
Consider the group homomorphism $\psi: V \to G/Z$ 
defined by $\psi(a)=aI_3Z\in G/Z$ for each class 
$a\in V\subseteq\bF_q^\times/U'\Det(G)$. Consider also the homomorphism 
$\vartheta:G/Z\to V$, given as the composition of the induced
determinant map on the quotient group, 
i.e., $\Det(xZ)=\Det(x)\Det(Z)\in\Det(G)/\Det(Z)$  for all $x\in G$, with
the quotient onto $\Det(G)/U'\Det(Z)\cong V$. 
We have $\vartheta\psi(a)=a^3$, which is an automorphism of $V$ because
$V$ is a $3'$-group. Since $\psi(V)$ is in the center of $G/Z$, we
conclude that $V$ is a direct factor of $G/Z$. So the first part of
the claim holds with $H$ the kernel of $\vartheta$.

For the last part of the statement, we observe that the
kernel of the restriction map $T(G/Z)\to T(H)$ is generated by the
isomorphism classes of indecomposable modules in the induction
$k_H^{\uparrow G/Z}$ of the trivial $kH$-module to $G/Z$. Since the
index $|G/Z:H|=|V|$ is not divisible by $3$ and the factor group $V$
is abelian, the induced module $k_H^{\uparrow G/Z}$ is a direct sum of
one-dimensional modules on which $H$ acts trivially. Therefore, the
kernel of the restriction map $T(G/Z)\to T(H)$ is isomorphic to
$X(V)\cong X(G/H)$ as required.
\end{proof}

The next result provides a description of $H$ and $V$ under our assumptions. 

\begin{prop} \label{prop:3structure2}
For $G$ and $Z$ as above, one 
of the two situations occurs. 
\begin{itemize}
\item[(a)] If $3$ does not divide $(q-1)/|\Det(G)|$, i.e., if $Z$ contains
  the Sylow $3$-subgroup of $Z(\GL(3,q))$, then 
$G/Z \cong \PGL(3,q) \times V$ where $V \cong \Det(G)/\Det(Z)$.
\item[(b)] Otherwise, $G/Z \cong \PSL(3,q) \times V$ where $V$ is the 
$3$-complement in $\Det(G)/\Det(Z)$.
\end{itemize}
In both cases, $T(G/Z) \cong T(H)\oplus X(G/Z)$, where 
$X(G/Z) \cong X(V)$, the group of one-dimensional $kV$-modules
and $H$ is either $\PGL(3,q)$ or $\PSL(3,q)$ as appropriate.  
\end{prop}

\begin{proof}
Suppose that $3$ does not divide $(q-1)/|\Det(G)|$.
By Lemma \ref{lem:3structure1} and its proof, we may assume that 
$\Det(G) = G/\SL(3,q)$ is a $3$-group. In the case that 
$\Det(G)$ is a Sylow $3$-subgroup of $\bfq^\times$, we must 
have $G/Z \cong \PGL(3,q)$, which proves (a). 

Otherwise, $\Det(G)$ is
not a Sylow $3$-subgroup of $\bfq^\times$, and so there exists an
element $\gamma \in \bfq^\times$, $\gamma \not\in \Det(G)$ such that $\gamma^3$
is in $\Det(G)$. Then the 
scalar matrix $X = \gamma I_3$ is an element of $G$ with the 
property that $\Det(X)$ generates $\Det(G)$, because $\bfq^\times$ is a
cyclic group. Since $Z$ contains the Sylow $3$-subgroup of $\Det(G)$, it
follows that $X \in Z$, and $Z\cdot\SL(3,q) = G$.
Hence, 
\[
G/Z \cong Z\cdot \SL(3,q)/Z \cong \SL(3,q)/(Z \cap \SL(3,q)) \cong \PSL(3,q),
\]
which proves (b).

The last statement, about the group of endotrivial modules, follows
because a complete set of nonisomorphic simple $kV$-modules all have
dimension one and define different blocks of $k(G/Z)$. 
Thus, any indecomposable endotrivial $k(G/Z)$-module is the (outer) 
tensor product of a one-dimensional $kV$-module and an
indecomposable endotrivial $kH$-module. Finally it should be 
noted that $H$ has a nontrivial normal $3$-subgroup, since, by
construction, $H$ is an extension of $Z\cdot \SL(3,q)/Z$
by a nontrivial $3$-group. Thus, $X(H)$ is trivial. 
\end{proof}

By Theorem \ref{thm:rank}, the torsion free rank of $T(G)$ is 
related to the number of conjugacy classes of maximal elementary 
abelian $3$-subgroups of rank $2$. The following calculation is 
important to determine $TF(G)$. 

\begin{prop} \label{prop:3structure3}
Let $G$ and $Z$ be as above.
The group $G/Z$ has $3$-rank $2$. In addition,
\begin{itemize}
\item[(a)] The group $\PGL(3,q)$ has three conjugacy 
classes of maximal elementary abelian $3$-subgroups.
\item[(b)] If $q \equiv 1 \pmod 9$ then $\PSL(3,q)$ has four conjugacy
classes of maximal elementary abelian $3$-subgroups. 
\item[(c)] If $q \equiv 4, 7 \pmod 9$ then a Sylow $3$-subgroup of 
$\PSL(3,q)$ is elementary abelian of order $9$.
\end{itemize}
\end{prop}

\begin{proof}
Write $q-1=3^td$ where $3$ does not divide $d$, and suppose that
$|Z|\geq3$.  A Sylow $3$-subgroup $S$ of $G = \GL(3,q)$ is generated
by elements 
\[
X_1=\begin{bmatrix}\zeta&&\\&1&\\&&1\end{bmatrix},
\quad
X_2=\begin{bmatrix}1&&\\&\zeta&\\&&1\end{bmatrix},
\quad
X_3=\begin{bmatrix}1&&\\&1&\\&&\zeta\end{bmatrix},
\quad
Y=\begin{bmatrix}&1&\\&&1\\1&&\end{bmatrix},
\]
where $\zeta$ is a primitive $3^t$ root of unity in $\bfq$.
Let $x_1, x_2, x_3$ and $y$ denote the images of 
$X_1, X_2, X_3$ and $Y$ (respectively) in $G = \PGL(3,q)$.
Note that $S \cap Z(\GL(3,q)) = Z(S) = \langle X_1X_2X_3 \rangle$.

Thus a Sylow $3$-subgroup of $G$ has a presentation
\[
S/Z(S) = \langle x_1,x_2,y~|~x_i^{3^t}=
y^3=1~,~\ls{y}{x_1}=x_1^{-1}x_2^{-1}~,
~\ls{y}{x_2}=x_1\rangle\cong(C_{3^t}\times
C_{3^t})\rtimes C_3
\]
The only central subgroup of order $3$ in $S/Z(S)$ is generated
by the element $x_1^rx_2^{-r}$ for $r = 3^{t-1}$. 

The $3$-group $S/Z(S)$ has rank $2$ and each noncyclic elementary
abelian subgroup has the form $\langle  z, x\rangle$ for some
noncentral element $x\in S/Z(S)$ of order $3$. Note that 
$(x_i^jy)^3=1$ for $i=1,2$ and any $0\leq j<3^t$.
Moreover the unique subgroup $C_3\times C_3$ in the normal subgroup
$C_{3^t}\times C_{3^t}$ of $S$ generated by $x_1$ and $x_2$ is 
characteristic in $S$. The other
maximal elementary abelian subgroups have the form
 $\langle z,x\rangle$ with
$x\not\in\langle x_1,x_2 \rangle$. All such elements
$x$ have order $3$. A routine calculation shows that there are
three $S$-conjugacy classes of these, namely
\[
\langle z, y\rangle, \ \ \ \ \   \langle z,x_1y\rangle\qbox{and}
\langle z,x_1y^2\rangle
\]
This is determined, for example, by looking at the 
monomial matrices obtained by conjugating $X_1Y$, and it gives 
us a total of four $S/Z(S)$-conjugacy classes of elementary abelian
subgroups of $S/Z(S)$ of order $9$.

In $\GL(3,q)$, with the above elements $X_1,X_2,X_3,Y$ we get
that
\[
X_1Y=\begin{bmatrix}&\zeta&\\&&1\\1&&\end{bmatrix}=\ls T{(X_1Y^2)}
\qbox{where}
T=\begin{bmatrix}1&&\\&&1\\&1&\end{bmatrix}.
\]
Thus $\langle Z,X_1Y\rangle$ is $G$-conjugate to
$\langle Z, X_1Y^2 \rangle$.  The same holds in $G=\PGL(3,q)$, 
that is, there are exactly three
$G$-conjugacy classes of $C_3\times C_3$.  This finishes the proof of (a)

For (b), we refer the reader to the results in \cite{DRV},
where $S$ is a $3$-group studied by N. Blackburn and denoted
$B(3,2t;0,0,0)$. Then the $3$-fusion system defined by $\PSL(3,q)$ on
$S$ stabilizes the three conjugacy classes of $3$-centric radical
elementary abelian subgroups of order $9$ of $S$, and there is a single
conjugacy class of elementary abelian subgroups of order $9$ that are
not $3$-centric radical in $S$. As a consequence, no two of the four conjugacy
classes of elementary abelian subgroups of $S$ of order $9$ fuse
in $G/Z$ (cf.~\cite[Theorem~5.10 and Tables 2 and 4]{DRV}). 

Finally, (c) is immediate from the observations that $9$ is the highest
power of $3$ wich divides $|\PSL(3,q)|$ for $q-1\equiv3\pmod9$ and that
$\PSL(3,q)$ has no element of order $9$.
\end{proof}

The following is the main result of the section.

\begin{thm} \label{thm:noncyclic3}
Suppose that $n= p =3$, and that $q$ a prime power such that $3$ 
divides $q-1$ (i.e., $e=1$). Let 
$\SL(3,q)\subseteq G\subseteq\GL(3,q)$ and $Z$ a
central subgroup of $G$ containing the 
Sylow $3$-subgroup of the center
$Z(G)$ of $G$.  Then the following hold.
\begin{itemize}
\item[(a)] If $3$ does not divide $(q-1)/|\Det(G)|$ 
then $T(G/Z) \cong {\bZ}^3\oplus X(G/Z)$.
\item[(b)] If $q \equiv 1 \pmod 9$ and if $3$ divides $(q-1)/|\Det(G)|$  
then $T(G/Z) \cong {\bZ}^4\oplus X(G/Z)$.
\item[(c)] If $q\equiv 4,7\pmod 9$ and if $3$ divides $(q-1)/|\Det(G)|$ 
then $T(G/Z)\cong {\bZ}\oplus{\bZ}/2{\bZ}\oplus 
{\bZ}/2{\bZ} \oplus X(G/Z)$.
\end{itemize}
In every case, $X(G/Z) \cong X(V)$ is the character
group of $V$, the normal $3$-complement in the cyclic group
$\Det(G)/\Det(Z)$. 
\end{thm}

\begin{proof}
The factors $X(G/Z)$ are determined in Proposition \ref{prop:3structure2}. 
The ranks of the torsion free parts of $T(G/Z)$ are established in 
Proposition \ref{prop:3structure3}. The only question is the kernel $K(G/Z)$
of the restriction $T(G/Z) \to T(S)$ where $S$ is a Sylow $3$-subgroup
of $G/Z$, which is isomorphic to either $\PGL(3,q)$ or $\PSL(3,q)$.  
In cases (a) and (b), we compute $K(G/Z)$ using Theorem \ref{thm:method}
with $H$ being the normalizer of the image of the torus in 
$\GL(3,q)$ or $\SL(3,q)$ respectively. Note here that the normalizer of the 
images of the torus is equal to the image of the normalizer of the 
torus. The calculation is very similar to that in the proofs of
Propositions \ref{prop:composite} and \ref{prop:sleps}, and we leave it
to the reader to fill in the details. 
The result is that $K(G/Z) = TT(G/Z) = \{0\}$ for $G/Z = \PGL(3,q)$ 
in case (a) and for $G/Z= \PSL(3,q)$ in case (b).

The only thing left is the calculation of $K(G/Z) = TT(G/Z)$ in case (c),
where $G/Z = \PSL(3,q)$ and $q \equiv 4, 7 \pmod 9$. In this situation, a Sylow 
$3$-subgroup $S$ of $G/Z$ is elementary abelian of order $9$ and the
methods of \cite{CT4} apply. More precisely, \cite[Theorem 8.4]{CT4}
shows that $K(G/Z) \cong \bZ/2\bZ \oplus \bZ/2\bZ$. 
\end{proof}

%%%%%%%%%%%%%%%%%%%%%%%  new section on on A2 char 2 %%%%%%%%%

\section{Type $A_2$ in Characteristic 2} \label{sec:sl32}

Throughout this section let $p=2$ and $G$ be a group such
that $\SL(3,q) \subseteq G \subseteq \GL(3,q)$. 
Let $Z$ be a central subgroup of $G$. Our objective is to determine
$T(G/Z)$ under these assumptions, namely addressing the first part of
the cases of Theorem \ref{thm:noncyclic1a} excluded by condition (d) of
the hypothesis.  
 
We begin with a decomposition of $G/Z$ (similar to that for $p=3$ and
Lemma~\ref{lem:3structure1}).

\begin{lemma} \label{lem:sl3char2}
Let $G$ and $Z$ be as given above. Then  $G/Z \cong H/Z_3 \times W_2
\times W$ where
\begin{itemize}
\item[(a)] $W$ is the direct product of the Sylow $\ell$-subgroups of
$\Det(G)/\Det(Z)$ for $\ell$ not equal to $2$ or $3$,
\item[(b)] $W_2$ is the Sylow $2$-subgroup of $\Det(G)/\Det(Z)$,
\item[(c)] $Z_3$ is the Sylow $3$-subgroup of $Z$, and
\item[(d)] $H$ is an extension
\[
\xymatrix{
1 \ar[r] & \SL(3,q) \ar[r] & H \ar[r] & V_3 \ar[r] &1
}
\]
where $V_3$ is the Sylow 3-subgroup of $\Det(G)/\Det(Z)$.
\end{itemize}
Moreover, $T(G/Z) \cong T(H/Z_3 \times W_2) \oplus X(W)$.
\end{lemma}

\begin{proof}
The proof follows the same line of reasoning as in  Lemma
\ref{lem:3structure1}.  That is, the composition
\[
Z(G) \to G \to \Det(G)
\]
induces an isomorphism from the Sylow $\ell$-subgroup of $Z(G)$
to the Sylow $\ell$-subgroup of $\Det(G)$ for all primes $\ell \neq 3$.
The same holds for the induced composition
\[
Z(G)/Z \to G/Z \to \Det(G)/\Det(Z).
\]
One can finish the proof by letting $H$ be the inverse image of the
Sylow 3-subgroup  of $\Det(G)/\Det(Z)$ under the map induced by the
determinant.
\end{proof}

Now we can state the main theorem of the section.

\begin{thm}  \label{thm:sl3char2} Let $G$ be a group such that
$\SL(3,q) \subseteq G \subseteq \GL(3,q)$ and let $Z \subseteq
  Z(G)$. Assume that the field $k$ has characteristic $2$. 
\begin{itemize}
\item[(a)] Suppose that $2$ divides the order of $\Det(G)/\Det(Z)$. Then
$T(G/Z) \cong \bZ \oplus X(G/Z)$.
\item[(b)] Suppose that $2$ does not divide the order of $\Det(G)/\Det(Z)$.
Then

\begin{itemize}
\item[(i)] if $4$ divides $q-1$, then $T(G/Z) \cong \bZ \oplus X(G/Z)$, 
\item[(ii)] if $4$ divides $q+1$, then $T(G/Z) \cong \bZ \oplus \bZ/2\bZ
\oplus X(G/Z)$.
\end{itemize}
\end{itemize}
\end{thm}

\begin{proof}
Let $H$, $Z_3$, $W_2$ and $W$ be as in Lemma
\ref{lem:sl3char2}. Suppose that $2$ divides the order of $\Det(G)/\Det(Z)$. 
Observe that $G/Z$ has $2$-rank $3$, since $H$ has $2$-rank $2$ (an easy fact
that is apparent later in this proof). Moreover, the center of
a Sylow $2$-subgroup of $G$ has $2$-rank $2$. So every maximal elementary
abelian subgroup has $2$-rank $3$. This means that $T(G/Z)$ has
torsion-free rank one. In addition, $G/Z$ has a nontrivial normal $2$-subgroup,
and hence every indecomposable $k(G/Z)$-module 
with trivial Sylow restriction has dimension one. The proves (a).

For the rest of the proof  assume that $2$ does not divide the order of $\Det(G)/\Det(Z)$.
That is, the group $W_2$ of Lemma~\ref{lem:sl3char2} is trivial.
Then, a Sylow $2$-subgroup of $G/Z$ is the image of a Sylow $2$-subgroup of $\SL(3,q)$
under the map $\SL(3,q) \to (Z \cdot \SL(3,q))/Z \hookrightarrow
G/Z$. Thus, $G/Z$ has
$2$-rank two and any two maximal elementary abelian $2$-subgroups are conjugate
in $\SL(2,q)$ and also in $G/Z$. It follows that $TF(G/Z)\cong {\mathbb Z}$ and 
generated by the class of $\Omega(k)$.  

Let $S$ denote the Sylow $2$-subgroup of $G/Z$. We distinguish two
cases. 
First, suppose that $q \equiv 1\pmod4$.
Then $q-1 = 2^td$ where $d$ is odd and $\vert S \vert = 2^{2t+1}$. We can
assume that $S$ is generated by the classes (modulo $Z$) of the elements
\[
X_1 = \begin{bmatrix}\zeta\\&\zeta^{-1} \\&& 1\end{bmatrix},
\quad
X_2 = \begin{bmatrix}1\\& \zeta\\&&\zeta^{-1}\end{bmatrix}, \ \text{and} \
\quad
Y = \begin{bmatrix}  & 1\\ -1\\&& 1\end{bmatrix}
\]
where $\zeta \in \bfq$ is a root of unity of order $2^t$.
The group $S$ is a wreath product and has a unique $G/Z$-conjugacy class
of Klein four subgroups. Therefore $TF(G/Z) \cong \bZ$.

For the other case, we suppose that $q \equiv 3\pmod4$ and that
$q+1 = 2^td$ where $d$ is odd and $\vert S \vert = 2^{t+2}$.
We can choose the Sylow $2$-subgroup $S$ of $G/Z$ that is the collection
of classes (modulo $Z$) of all block matrices of the form
\[
\begin{bmatrix} X  & \\ \\& r \end{bmatrix}
\]
where $r = \Det(X)^{-1}$ and 
where $X$ runs through the elements of some fixed Sylow $2$-subgroup of $\GL(2,q)$. Thus,
$S$ is isomorphic to a Sylow $2$-subgroup of $\GL(2,q)$, since the
two groups have the same order. It is well known that $S$ is
semi-dihedral (see also \cite{ABG, CF}). Hence, 
$T(S) \cong \bZ \oplus \bZ/2\bZ$. By \cite{CMT2}, the restriction
map $T(G/Z) \to T(S)$ is a split surjection.

The only issue left to prove is that every indecomposable $k(G/Z)$-module
with trivial Sylow restriction has dimension one. Observe
by Lemma \ref{lem:sl3char2}, we may assume that $G = H$ is an
extension of $\SL(3,q)$ by a cyclic group whose order is a power
of $3$. With this assumption $G/Z = H/Z_3$.
The asserted result is obtained in two steps.

Let $J = (Z_3\cdot \SL(3,q))/Z_3 \cong \SL(3,q)/(Z_3 \cap
\SL(3,q))$. The first step is to show that $T(J) = T(S)$, or equivalently,
that every indecomposable endotrivial $kJ$-module with trivial Sylow
restriction has dimension one. 

Notice that $Z_3 \cap \SL(3,q)$ is in the center of
$\SL(3,q)$ and has order either $1$ or $3$. The normalizer of the Sylow
$2$-subgroup $S$ of $J$ has the form $N_J(S) = S \times Z(J)$ with
$Z(J)$ of order $1$ or $3$.
Suppose that $|Z(J)|=1$, that is, either $Z_3 \neq \{1\}$ or $3$ does
not divide $q-1$. Then $T(J) = T(S)$ by
Proposition~\ref{prop:selfnormal} and the first step is complete in this case.

Now assume that $Z(J)$ has order $3$. Then $3$ divides $q-1$,
$Z_3 = \{1\}$ and $J = \SL(3,q)$. Let $u$ be a primitive cube root of one.
We fix the following elements of $J$
\[
X = \begin{bmatrix}u\\& u^{2} \\&& 1\end{bmatrix}, \quad
Y = \begin{bmatrix}1 \\& u^2 \\&& u \end{bmatrix}, \quad
A = \begin{bmatrix} -1\\& -1 \\&& 1\end{bmatrix},  
\quad B = \begin{bmatrix} 1\\& -1 \\&& -1\end{bmatrix}.
\]
Note that $XY$ generates the Sylow $3$-subgroup of $Z(J)$. 
Moreover, $X$ is in the commutator subgroup (which is isomorphic to
$\SL(2,q)$) of $N_J(\langle A \rangle)$ and, similarly, $Y$ is in the
commutator subgroup of $N_J(\langle B \rangle)$. Furthermore,
$S\subseteq\rho^2(\langle A\rangle)$ by an argument similar to the one
used in the proof of Proposition~\ref{prop:rp-case}. So
in the notation of Section \ref{sec:maintec},
$X \in \rho^1(\langle A \rangle)$ and  $Y \in \rho^1(\langle B \rangle)$.
It follows that $Y \in \rho^2(\langle A \rangle)$ and
$XY \in \rho^3(S)$. Thus, $\rho^3(S) = N_J(S)$, and  
from \cite{CT4} we have that the only indecomposable endotrivial
$kJ$-module with trivial Sylow restriction is the trivial module. This completes
the first step.

For the second step and to finish the proof of the theorem, suppose that
$M$ is a $k(G/Z)$-module with trivial Sylow restriction. Then
$M_{\downarrow J} \cong k \oplus \proj$. 
This implies that $M$ is a direct summand
of $(k_{J})^{\uparrow G/Z}$ which is a direct sum of one-dimensional 
modules, since $J$ is a normal subgroup of $G/Z$ of odd index.
Therefore, $M$ has dimension one.
\end{proof}

\section{Type $A_1$ in Characteristic 2} \label{sec:sl22}

Throughout this section let $p=2$ and let $G$ be a group such
that $\SL(2,q) \subseteq G \subseteq \GL(2,q)$. 
Let $Z$ be a central subgroup of $G$. Our objective is to determine
$T(G/Z)$ under these assumptions, namely addressing the second part of
the cases of Theorem \ref{thm:noncyclic1a} excluded by condition (d) of
the hypothesis.

This case is more tedious than the previous one because the group $G/Z$
can have dihedral (including Klein four), semi-dihedral or generalized quaternion
Sylow $2$-subgroups. The differences in the $2$-local structure of $G/Z$
lead to as many distinct outcomes for the structure of $T(G/Z)$, which we
now detail.

As in the previous two sections, we start with a useful decomposition of
$G/Z$. The proof of the lemma below is similar to that of Lemma
\ref{lem:sl3char2} and therefore left to the reader. 

\begin{lemma} \label{lem:sl2char2}
For $G$ and $Z$ as above, $G/Z \cong H/Z_2\times W$ where
\begin{itemize}
\item[(a)] $W$ is the odd part of $\Det(G)/\Det(Z)$,
\item[(b)] $Z_2$ is the Sylow $2$-subgroup of $Z$, and
\item[(c)] $H$ is an extension
\[
\xymatrix{
1 \ar[r] & \SL(2,q) \ar[r] & H \ar[r] & V_2 \ar[r] &1
}
\]
where $V_2$ is the Sylow $2$-subgroup of $\Det(G)/\Det(Z)$.
\end{itemize}
In addition  $T(G/Z) \cong T(H/Z_2) \oplus X(W)$, with $X(H/Z_2) =
\{1\}$ and $X(G/Z) \cong X(W)$.
\end{lemma}

We can now state and prove the main theorem of this section.

\begin{thm}  \label{thm:sl2char2} Let $G$ be a group such that
$\SL(2,q) \subseteq G \subseteq \GL(2,q)$ 
and $Z \subseteq Z(G)$. Assume that the field
$k$ has characteristic 2. Write $\vert G: \SL(2,q) \vert = 2^se$ and
$\vert Z \vert = 2^rf$ where $e$ and $f$ are odd integers.
The group of endotrivial modules
$T(G/Z)$ is described as follows.
\begin{itemize}
\item[(A)] Suppose that $q \equiv 3$ (mod  4). Write $q+1 = 2^td$ where
$d$ is odd. Note that $r, s  \in \{0, 1\}$.
\begin{itemize}
   \item[(1)] Assume that $s = 0$.
   \begin{itemize}
      \item[(a)] If $r = 0$, then $T(G/Z) \cong \bZ/4\bZ \oplus \bZ/2\bZ
            \oplus X(G/Z).$
      \item[(b)] If $r = 1$, then
      \begin{itemize}
         \item[(i)] if $q \equiv 3$ (mod 8), then
                $T(G/Z) \cong \bZ \oplus \bZ/3\bZ \oplus X(G/Z),$
         \item[(ii)] if $q \equiv 7$ (mod 8), then
                $T(G/Z) \cong \bZ^{2} \oplus X(G/Z).$
       \end{itemize}
   \end{itemize}
   \item[(2)] Assume that $s = 1$.
   \begin{itemize}
      \item[(a)] If $r = 0$ then
            $T(G/Z) \cong \bZ \oplus \bZ/2\bZ \oplus X(G/Z).$
      \item[(b)] If $r = 1$ then
            $T(G/Z) \cong \bZ^{2} \oplus X(G/Z).$
   \end{itemize}
\end{itemize}
\item[(B)] Suppose that $q \equiv 1$ (mod  4). Write $q-1 = 2^td$ where
$d$ is odd. Note that $0 \leq r \leq s+1$, $r \leq t$ and $s \leq t$.
\begin{itemize}
   \item[(1)] Assume that $r = 0$.
   \begin{itemize}
      \item[(a)] If $s = 0$, then
            $T(G/Z) \cong \bZ/4\bZ \oplus \bZ/2\bZ \oplus X(G/Z).$
      \item[(b)] If $s > 0$, then
            $T(G/Z) \cong \bZ \oplus X(G/Z).$
   \end{itemize}
   \item[(2)] Assume that $r > 0$.
   \begin{itemize}
      \item[(a)] If $0 < r < s+1 \leq t$, then
            $T(G/Z) \cong \bZ \oplus X(G/Z).$
      \item[(b)] If $r = s+1 \leq t$, then
      \begin{itemize}
         \item[(i)] if $q \equiv 1$ (mod 8), then
                $T(G/Z) \cong \bZ^{2} \oplus X(G/Z),$
         \item[(ii)] if $q \equiv 5$ (mod 8), then
                $T(G/Z) \cong \bZ \oplus \bZ/3\bZ \oplus X(G/Z).$
      \end{itemize}
      \item[(c)] If $r = s = t$, then
                $T(G/Z) \cong \bZ^{2} \oplus X(G/Z).$
   \end{itemize}
\end{itemize}
\end{itemize}
\end{thm}

\begin{proof}
For the purposes of the proof let $H$ and $Z_2$ be as in Lemma
\ref{lem:sl2char2}. Hence, it suffices to find
$T(H/Z_2)$ in each of the above cases. 

Let $K$ denote the kernel of the restriction map $T(H/Z_2) \to 
T(S)$, where $S$ is a Sylow $2$-subgroup of $H/Z_2$. We should first
note that if $r = 0$ or if $r < s+1\leq t$ or if $r < s =t$, then
$K = \{1\}$, and the restriction map is injective. The reason is that 
in each of these cases $H/Z_2$ has a nontrivial central $2$-subgroup and
$X(H/Z_2)=\{1\}$, since $\SL(2,q)$ is a perfect group. 
Thus $K=\{1\}$ by Proposition \ref{prop:normal} and Lemma
\ref{lem:sl2char2}. It follows that the only cases in which $K$ might
not be trivial are (A)(1)(b), (B)(1)(b), and (B)(2)(c). 

Suppose first that $q+1 = 2^td$ for $t>1$ and $d$ odd. 
A Sylow $2$-subgroup of 
$\GL(2, q)$ is a semi-dihedral group of order $2^{t+2}$ and it is
self-normalizing, by \cite{ABG, CF}. 
This is the Sylow $2$-subgroup in the case (A)(2)(a). 
So the restriction map $T(H/Z_2) \to T(S) \cong\bZ \oplus \bZ/2\bZ$ 
is an isomorphism by \cite{CMT2}. 
In case (A)(1)(a), a Sylow $2$-subgroup of $H/Z_2$ is generalized
quaternion, as for $\SL(2,q)$, and so the restriction map 
$T(H/Z_2) \to T(S) \cong \bZ/4\bZ \oplus \bZ/2\bZ$ is an isomorphism
by \cite{CMT2}.

If $r=1$, the group $Z_2$ is the center of the Sylow 
$2$-subgroup of $H$. Thus, the Sylow $2$-subgroup $S$ of $H/Z_2$ is a 
dihedral group, possibly a Klein four group. In cases (A)(1)(b)(ii), and 
(A)(2)(b), $S$ is dihedral of order at least $8$. 
In these cases the group $H/Z_2$ has two conjugacy classes of (maximal)
elementary abelian $2$-subgroups. Hence, the torsion-free rank of $T(H/Z_2)$
is two, by Theorem \ref{thm:rank}. Note also that $S$ is self-normalizing.
Thus by Proposition \ref{prop:selfnormal}, $T(H/Z_2) \cong {\bZ}^{\times2}$
as asserted.   

In the case (A)(1)(b)(i), $H/Z_2 \cong \PSL(2,q)$, 
and $S$ is a Klein four group with normalizer 
$N_{H_2/Z}(S)\cong S\rtimes C_3$ of order $12$. The Green 
correspondents of the nontrivial $kN_{H/Z_2}(S)$-modules of dimension one
are $k(H/Z_2)$-modules with 
trivial Sylow restriction of dimension greater than
one. The detailed computation of $T(H/Z_2)$ is carried out in \cite{CT4}.

Suppose now that $q-1 = 2^td$ for $t > 1$ and $d$ odd. 
A Sylow $2$-subgroup of $\GL(2,q)$ is a wreath product, which we can
choose to be generated by 
\[
\begin{bmatrix} 0 & 1 \\ 1 & 0 \end{bmatrix}, \quad 
\begin{bmatrix} \zeta & 0 \\ 0 & 1 \end{bmatrix} \quad \text{and} \quad 
\begin{bmatrix} 1 &0 \\ 0 & \zeta \end{bmatrix}
\]
where $\zeta$ is a $2^t$-root of unity in $\bF_q^\times$. 
A Sylow $2$-subgroup of $\SL(2,q)$ is a generalized quaternion group
\cite{ABG}. Hence, if $r = s = 0$, then we have the same situation as in
case (A)(1)(a). 
If $r = 0 < s$, then the subgroup consisting of diagonal matrices with
entries $1$ and $-1$, has rank $2$ and every involution is
$H/Z_2$-conjugate to an element of this subgroup. Consequently, $H/Z_2$
has a unique conjugacy class of maximal elementary abelian $2$-subgroups,
all of which have order 4. 
Hence, the torsion-free rank of $T(G)$ is 
one and the proof of (B)(1) is complete. 

If $0 < r < s+1 \leq t$, then the group $S$ has an elementary abelian subgroup
of rank $3$, generated by the classes (modulo $Z_2$) of the elements 
\[
\begin{bmatrix} 0 & 1 \\ 1 & 0 \end{bmatrix}, \quad 
\begin{bmatrix} 1 & 0 \\ 0 & -1 \end{bmatrix} \quad \text{and} \quad 
\begin{bmatrix} \zeta &0 \\ 0 & \zeta \end{bmatrix}
\]
where $\zeta$ is a $2^{s+1}$-root of $1$ in $\bF_q^\times$. In addition,
the last two elements above are central in $S$ and hence $S$ has no
maximal elementary abelian subgroups of rank $2$. Because $H/Z_2$ has a
nontrivial normal $2$-subgroup, any indecomposable $k(H/Z_2)$-module 
with trivial Sylow restriction has dimension $1$ and the claim holds in
case (B)(2)(a).

If $r = s+1 \leq t$ then the composition 
\[
\SL(2, q) \to  (Z_2 \cdot \SL(2, q))/Z_2  \hookrightarrow H/Z_2 
\]
is surjective and has kernel $Z(\SL(2, q))$. Thus $H/Z_2 \cong \PSL(2,q)$.
Its Sylow $2$-subgroup is a dihedral group or, in the case that $q \equiv 
5\pmod 8$, a Klein four group. Hence, we have the 
same situation as in (A)(1)(b) with the same result. In particular,
the results of \cite{CT4} apply in case (B)(2)(b)(ii). 

Finally, in case (B)(2)(c), a Sylow $2$-subgroup $S$ of $H/Z_2$ is
isomorphic to the quotient of a Sylow $2$-subgroup of $\GL(2,q)$ by its
center. So $S$ is a dihedral group of order at least $8$
(cf. \cite{CF}). Hence, the conclusion is the same as in case
(A)(1)(b)(ii).  
\end{proof}

\section{Appendix: Classification of Endotrivial Modules in 
the Cyclic Sylow Subgroup Setting}\label{sec:appendix}

The following result summarizes the main results of \cite{CMN3} and
provides a classification of the group of endotrivial 
modules for finite groups of Lie type $A$ 
in the case when a Sylow $p$-subgroup of $G$ is cyclic. 

\begin{thm} \label{thm:cyclic}
Suppose that $\SL(n,q) \subseteq G \subseteq \GL(n,q)$ and that 
$Z \subseteq Z(G)$. Assume that the Sylow $p$-subgroup $S$ 
of $G$ is cyclic and let $N = N_G(S)$. 
Then $T(G/Z) \cong T(\whn)$ where $\whn = N_{G/Z}(\whs)$ and 
$\whs$ is a Sylow $p$-subgroup of $G/Z$. 
Moreover, $T(\whn)$ is the middle term of a not necessarily split
extension 
\begin{equation} \label{eq:sequence}
\xymatrix{
1 \ar[r] & X(\whn) \ar[r] & T(\whn) \ar[r] & T(\whs) \ar[r] & 0
}
\end{equation} 
where $X(\whn)\cong N/(Z[N,N])$ is the 
group of isomorphism classes of
$k\whn$-modules of dimension one.
Let $D = \Det(G) \cong G/\SL(n,q)$ and let 
$d = \vert D \vert$. In the case that $Z = \{1\}$ we have the 
following. 
\begin{itemize}
\item[(a)] If $p=2$ then $n=1$, and $T(G) \cong D/\Det(S)$.
\item[(b)] Suppose that $p>2$ divides $q-1$. If $p$ divides $d$,
then $n = 1$ and $T(G) \cong \bZ/a\bZ \oplus \bZ/2\bZ$,
where $d = ap^t$ for $a$ relatively prime to $p$.
\item[(c)] If $p>2$ divides $q-1$ and $p$ does not divide $d$, 
then there are two possibilities:
\begin{itemize}
\item[(i)] assuming that~2 does not divide $(q-1)/d$, then 
$T(G) \cong \bZ/d\bZ \oplus \bZ/4\bZ$. 
\item[(ii)] assuming that~2 divides $(q-1)/d$, then $
T(G) \cong \bZ/d\bZ \oplus \bZ/4\bZ \oplus \bZ/2\bZ$. 
\end{itemize}
\item[(d)] Suppose that $p$ does not divide $q-1$. Let $e$ be the least 
integer such that $p$ divides $q^e-1$. Then $n = e+f$ for some 
$f$ with $0 \leq f < e$. Let $m = (q-1)/d$ and $\ell = 
\gcd\big(m(q-1),q^e-1\big)/m$. Then we have two possibilities: 
\begin{itemize}
\item[(i)] if $f=0$ then $T(G) \cong \bZ/\ell\bZ \oplus \bZ/2e\bZ$,
\item[(ii)] while if $f >0$, then  
$T(G) \cong \bZ/2e\bZ \oplus \bZ/(q-1)\bZ \oplus \bZ/d\bZ$ 
(except that $T(G) \cong \bZ/2e\bZ \oplus 
\bZ/2\bZ$ if both $f=2$ and $q = 2$).
\end{itemize}
\end{itemize}
\end{thm}

\end{document}